\documentclass[12pt]{amsart}
\usepackage{amssymb,amsthm,enumerate}
\usepackage{eucal,mathrsfs,esvect}
\usepackage[bbgreekl]{mathbbol}
\usepackage{color}
\usepackage[all]{xy}
\usepackage{tikz-cd}
\usepackage{mathabx}
\usepackage{hyperref}
\DeclareMathAlphabet{\mathcalligra}{T1}{calligra}{m}{n}
\usepackage[normalem]{ulem}
\usepackage{calc}
\newtheorem{lem}{Lemma}[section]
\newtheorem{thm}[lem]{Theorem}
\newtheorem*{thm*}{Theorem}

\newtheorem{prop}[lem]{Proposition}

\theoremstyle{definition}
\newtheorem{defn}[lem]{Definition}
\newtheorem{rem}[lem]{Remark}

\usepackage{baskervald}
\makeatletter
\newcommand*\@dblLabelI {}
\newcommand*\@dblLabelII {}
\newcommand*\@dblequationAux {}

\def\@dblequationAux #1,#2,%
    {\def\@dblLabelI{\label{#1}}\def\@dblLabelII{\label{#2}}}

\newcommand*{\doubleequation}[3][]{%
    \par\vskip\abovedisplayskip\noindent
    \if\relax\detokenize{#1}\relax
       \let\@dblLabelI\@empty
       \let\@dblLabelII\@empty
    \else 
       \@dblequationAux #1,%
    \fi
    \makebox[0.4\linewidth-1.5em]{%
     \hspace{\stretch2}%
     \makebox[0pt]{$\displaystyle #2$}%
     \hspace{\stretch1}%
    }%
    \makebox[0.2\linewidth]{\ }
    \makebox[0.4\linewidth-1.5em]{%
     \hspace{\stretch1}%
     \makebox[0pt]{$\displaystyle #3$}%
     \hspace{\stretch2}%
    }%
    \makebox[3em][r]{(%
  \refstepcounter{equation}\theequation\@dblLabelI,
  \refstepcounter{equation}\theequation\@dblLabelII)}%
  \par\vskip\belowdisplayskip
}
\makeatother

\newcommand{\mbb}[1]{\mathbb #1}

\newcommand{\mc}[1]{\mathcal #1}
\newcommand{\ms}[1]{\mathscr #1}

\newcommand{\oper}[1]{\operatorname{#1}}

\newcommand{\Br}{\oper{Br}}

\newcommand{\ind}{\oper{ind}}

\newcommand{\Gal}{\oper {Gal}}

\newcommand{\Hom}{\oper{Hom}}

\newcommand{\cha}{\oper{char}}
\newcommand{\Spec}{\oper{Spec}}

\newcommand{\cd}{\oper{cd}}

\newcommand{\ov}{\overline}
\newcommand{\til}{\widetilde}
\newcommand{\wh}{\widehat}

\DeclareMathOperator{\res}{res}

\newcommand{\spdim}[3]{\operatorname{sd}_{#2}^{#1}(#3)}
\newcommand{\stspdim}[3]{\operatorname{ssd}_{#2}^{#1}(#3)}
\newcommand{\gstspdim}[3]{\operatorname{gssd}_{#2}^{#1}(#3)}

\def\<{\left<}
\def\>{\right>}

\newcommand{\thing}[2]{
  {\xymatrix{#1 \ar@<3pt>[r] \ar@<-3pt>[r] & #2 }}}
\newcommand{\things}[4]{
  {\xymatrix{#1 \ar@<3pt>[r]^{#3} \ar@<-3pt>[r]_{#4} & #2 }}}

\usepackage{fullpage}

\title{Bounding cohomology classes over semiglobal fields
}
\author{David Harbater}
\author{Julia Hartmann}
\author{Daniel Krashen}

\date{\today}

\thanks{
\textit{Mathematics Subject Classification} (2020): 11E72, 12G05, 14H25 (primary); 14G27, 12E30 (secondary). \\
\textit{Key words and phrases.} Galois cohomology, splitting, arithmetic curves, local-global principles, finite group schemes, semi-global fields, patching, \'etale algebras.
}

\begin{document}

\maketitle
\begin{center}
{\em \small Dedicated to Moshe Jarden on his 80th birthday, to honor his contributions to patching methods in algebra.}
\end{center}

\begin{abstract}We provide a uniform bound for the index of cohomology classes in $H^i(F, \mu_\ell^{\otimes i-1})$ when $F$ is a semiglobal field (i.e., a one-variable function field over a complete discretely valued field $K$). The bound is given in terms of the analogous data
for the residue field of $K$ and its finitely generated extensions of
 transcendence degree at most one. We also obtain analogous bounds for collections of cohomology classes. Our results provide recursive formulas for function fields over higher rank complete discretely valued fields, and explicit bounds in some cases when the information on the residue field is known.  In the process, we prove a splitting result for cohomology classes of degree $3$ in the context of surfaces over finite fields.
\end{abstract}

\section{Introduction}

It is classical that the index of a central simple algebra over a global field $F$ is equal to its period as an element of the Brauer group.  In terms of Galois cohomology, this says that any element of $H^2(F, \mu_n)$ is split by an extension of degree $n$ over $F$.  The corresponding assertion does not generally hold for other fields $F$, though the period always divides the index, and the index always divides some power of the period (\cite{Pie82}, Proposition~14.4(b)(ii)). In \cite{Salt97} (see also \cite{Salt98}), it was shown
that for a one-variable function field\footnote{In this paper, we use the term {\em one-variable function field} $F$ over a field $K$ to mean a finitely generated extension of $K$ of transcendence degree one; we do not require $K$ to be algebraically closed in~$F$.} $F$ over $\mbb Q_p$, the index divides the square of the period, provided that the period is prime to $p$.

More generally, given a field $F$, one can ask if there is a uniform bound on the index in terms of the period, that is, whether there is an integer $d$ such that the index of every central simple $F$-algebra divides the $d$-th power of its period.
Starting with \cite[page~12]{CTnotes} (see also
\cite{LiebTwist}), the idea has emerged that for large classes of fields, such a uniform bound $d$ should exists, and that it should increase by one upon passage to one-variable function fields. So far, there have been a number of results giving such bounds and giving evidence for this idea.
In the case that~$F$ is a one-variable function field over a complete discretely valued field with residue field~$k$, and the period is prime to $\cha(k)$, such a bound $d$ for $F$ was found in \cite{Lieb} and \cite{HHK} in terms of the corresponding bounds for fields that are extensions of $k$ that are either finite or finitely generated of transcendence degree one.  This generalized \cite{Salt97}.
More recently, for such a field~$F$, a bound was found for a ``simultaneous index'' in \cite{gosavi2019generalized}; i.e., for the degree of an extension of $F$ that simultaneously splits an arbitrary finite set of $\ell$-torsion Brauer classes over~$F$, for a given prime $\ell \ne \cha(k)$.

In this paper, we focus on higher degree Galois cohomology groups $H^i(F,\mu_n^{\otimes i-1})$, for $i >2$. 
These higher cohomology groups have already been the subject of much investigation from various perspectives. We note in particular that in \cite{Kato} these were viewed in certain contexts as generalizations of the $n$-torsion subgroup $H^2(F, \mu_n)$ of the Brauer group $\Br(F)$ for $F$ a higher dimensional local or global field. However, much less is known in general about uniform period-index bounds for these groups; and although some conjectures have been made (see for example \cite[Conjecture~1, page 997]{DK-period}), supporting evidence has been difficult to obtain. Some important progress has been made in the case of degree~$3$ cohomology, showing that period and index coincide in the case of function fields of $p$-adic curves (\cite{PS}), function fields of surfaces over finite fields (\cite{PS3}), and more recently for function fields of curves over imaginary number fields \cite{Suresh}. Motivated by Kato's work, by the results on Brauer groups, as well as these results for degree~$3$ cohomology, in this paper we study the problem of bounding the index of a class in $H^i(F,\mu_\ell^{\otimes i-1})$ in terms of its period $\ell$, where $F$ is a one-variable function field over a complete discretely valued field $K$ with residue field $k$; and more generally bounding the minimal degree of an extension of $F$ that simultaneously splits finitely many such classes.
Namely, we define $\stspdim{i}{\ell}{F}$, called the stable $i$-splitting dimension at $\ell$ of $F$, to be the minimal $d$ such that for all finite extensions $L/F$, and for all $\alpha \in H^i(L,\mu_\ell^{\otimes i-1})$, $\ind(\alpha)$ divides $\ell^d$.  We similarly define the generalized stable $i$-splitting dimension at $\ell$ of $F$ to be an analogous quantity $\gstspdim{i}{\ell}{F}$ for the simultaneous splitting of finite sets of elements $B\subseteq H^i(L,\mu_\ell^{\otimes i-1})$.  In Theorem~\ref{main thm}, we show the following generalization of the main theorem in \cite{gosavi2019generalized}:

\begin{thm*}
In the above situation, $\stspdim{i}{\ell}{F} \le \stspdim{i}{\ell}{k} + \stspdim{i}{\ell}{k(x)} + \varepsilon$, where $\varepsilon=2$ if $\ell$ is odd and $\varepsilon=3$ if $\ell=2$.  The analogous bound also holds for $\gstspdim{i}{\ell}{F}$. Here $i$ is any positive integer.
\end{thm*}

Our approach first reduces to the case of unramified classes using a splitting result of 
\cite{gosavi2019generalized}.  The proof in the unramified case relies on patching over fields, a framework introduced in \cite{HH} (which was also used in \cite{HHK} and \cite{gosavi2019generalized}). In particular, it relies on a local-global principle for Galois cohomology from \cite{Hi}. 
In the case when $i=2$, i.e., when considering classes in the Brauer group, our bound agrees with that given in \cite{gosavi2019generalized} for collections of Brauer classes, but it is weaker than the bound given in \cite{HHK} for a single Brauer class. The main theorem implies recursive bounds for function fields over higher rank complete discretely valued fields.
In the final section of this paper, we apply our results in specific situations to obtain explicit numerical bounds for $\stspdim{i}{\ell}{F}$ and $\gstspdim{i}{\ell}{F}$.  These bounds give information on degree~$3$ and higher cohomology classes, in cases when the information on the Brauer group is not sufficient to obtain bounds with prior methods. For example, if $F$ is a one-variable function field over a complete discretely valued field whose residue field is a 
global function field
and $\ell$ is odd, then $\gstspdim{3}{\ell}{F}$ is at most $3$; see Proposition~\ref{number field ex}.  In order to obtain these numerical bounds, we prove a splitting result for 
surfaces over a finite field (Theorem~\ref{deg ell splitting}), which should be of independent interest. Both the splitting result and the applications 
build on work of Kato (see \cite{Kato}).

We thank the anonymous referee for helpful comments that led to improvements and to simplifications of some of the arguments.

\section{Uniform bounds for cohomology classes}

In this section, we define quantities that bound the degree of extensions needed to split a cohomology class, or a finite collections of such classes.

\begin{defn} \label{spl def}
Let $F$ be a field, and fix a prime $\ell\neq \operatorname{char}(F)$ and a positive integer $i$. 
A field extension $L/F$ is called a {\it splitting field} for a class  $\alpha \in H^i(F, \mu_\ell^{\otimes i - 1})$, if the image $\alpha_L$ of $\alpha$ under the natural map  $H^i(F, \mu_\ell^{\otimes i - 1})\rightarrow  H^i(L, \mu_\ell^{\otimes i - 1})$ is trivial. In that case, we also say that $\alpha$ {\em splits over $L$}.
Similarly, if $B \subseteq H^i(F, \mu_\ell^{\otimes i - 1})$ is a collection of elements, we say that a field extension $L/F$ is a {\em splitting field for $B$} if it is splitting field for each element of $B$.

The {\it index of a class} $\alpha \in H^i(F, \mu_\ell^{\otimes i - 1})$, denoted  by $\operatorname{ind}(\alpha)$, is the greatest common divisor of the degrees of splitting fields of $\alpha$ that are finite over $F$. Similarly, the {\em index of a subset} $B \subseteq H^i(F, \mu_\ell^{\otimes i - 1})$ is the greatest common divisor of the degrees of splitting fields of $B$ that are finite over $F$.
\end{defn}

\begin{rem}\label{rescores}
We will frequently use that if $\alpha\in H^i(F, \mu_\ell^{\otimes i - 1})$ and $E/F$ is a finite field extension of degree prime to $\ell$ such that $\alpha_E$ is trivial, then $\alpha$ is trivial, by a standard restriction-corestriction argument (using that the composition of restriction and corestriction is multiplication by the degree).
\end{rem}

\begin{lem}\label{l-power}
For $F$ a field, $\ell\neq \operatorname{char}(F)$ a prime, and $i$ a positive integer, let
$\alpha \in H^i(F, \mu_\ell^{\otimes i - 1})$. Then there exists a splitting field $L/F$ so that $[L:F]$ is a power of $\ell$. In particular, the index of $\alpha$ is a power of $\ell$. More generally, the index of a finite subset $B \subseteq H^i(F, \mu_\ell^{\otimes i - 1})$ is a power of $\ell$.
\end{lem}

\begin{proof}
Let $\rho$ be a primitive $\ell$-th root of unity, and let $\til F:=F(\rho)$.  By the Bloch-Kato conjecture/norm residue isomorphism theorem (\cite[Theorem~6.16]{voevodsky}; see also \cite{Wei}), $\alpha_{\til{F}}\in H^i(\til{F}, \mu_\ell^{\otimes i - 1})\cong H^i(\til{F}, \mu_\ell^{\otimes i})$ may be written as a sum of symbols. That is, $\alpha_{\til{F}}=\sum_{j=1}^m \beta_j$, where $\beta_j=(b_{j1})\cup \cdots \cup (b_{ji})$ for elements $b_{jk}\in {\til F}^\times$; here for $b\in \til{F}^\times$, $(b)$ denotes the class in $H^1(\til{F},\mu_\ell)\cong {\til{F}}^\times/({\til{F}}^\times)^\ell$. It then follows that $E:=\til{F}(\sqrt[\ell]{b_{11}},\ldots,\sqrt[\ell]{b_{m1}})$ is a splitting field for $\alpha$ (see also \cite{DK-period}, Remark~2.3). Let $\til{E}$ be the Galois closure of $E/F$. Note that $\til{E}/\til{F}$ is a compositum of cyclic (Galois) extensions of prime degree $\ell$ (viz., those obtained by adjoining $\ell$-th roots of the $\operatorname{Gal}(\til{F}/F)$-conjugates of the elements $b_{jk}$). Hence $\operatorname{Gal}(\til{E}/\til{F})$ is a subdirect product of cyclic groups of order $\ell$ (see, e.g. \cite{Dummit}, Chap.~14, Proposition~21). By induction, one checks that such a subdirect product is in fact a direct product of cyclic groups of order $\ell$, using that for $H_1$ cyclic of order $\ell$ and $H_2$ of $\ell$-power order, $H_1\cap H_2$ is either equal to $H_1$ or trivial. Thus $\operatorname{Gal}(\til{E}/\til{F})$ is an (elementary abelian) $\ell$-group. By the Schur-Zassenhaus theorem (\cite{Zassenhaus}, IV.7. Theorem 25; or \cite{Suzuki}, Chap. 2, Theorem~8.10),
$\operatorname{Gal}(\til{E}/F)$ contains a subgroup of $\ell$-power index and order $[\til{F}:F]$ dividing $\ell-1$. Its fixed field is an extension $L/F$ of $\ell$-power order. Since $\til{E}/L$ is of degree prime to $\ell$ and $\til{E}$ is a splitting field of $\alpha$, so is $L$  (Remark~\ref{rescores}), proving the first assertion. Note that the same argument applies to finite collections of cohomology classes. The statements on the index are immediate consequences.
\end{proof}

As a consequence of the above lemma, we can make the following definition.

\begin{defn}
For a prime $\ell$ and a positive integer $i$, 
we say that the {\it $i$-splitting dimension at} $\ell$ of $F$, denoted by $\spdim{i}{\ell}{F}$, is the minimal exponent $n$ so that  $\operatorname{ind}(\alpha) \mid \ell^n$ for all $\alpha \in H^i(F, \mu_\ell^{\otimes i - 1})$.
\end{defn}

We would like to show that the splitting dimension behaves in a controlled way upon finitely generated extensions of certain fields, and with respect to complete fields and their residues. In order to facilitate this, we will use a stronger form of splitting dimension, to require stability under finite extensions. This is analogous to notions introduced for quadratic forms and central simple algebras in \cite{HHK}.

\begin{defn}
Let $i$ be a positive integer. We say that the {\it stable $i$-splitting dimension at $\ell$ of $F$}, denoted $\stspdim{i}{\ell}{F}$, is the minimal $n$ so that $\spdim{i}{\ell}{E} \leq n$ for all finite field extensions $E/F$.
\end{defn}

In analogy to \cite{gosavi2019generalized}, we also consider collections of cohomology classes.
\begin{defn}Let $i$ be a positive integer. We define the {\em generalized stable $i$-splitting dimension} of a field $F$, denoted by $\gstspdim{i}{\ell}F$, to be the minimal exponent $n$ so that $\operatorname{ind}(B) \mid \ell^n$ for all finite field extensions $E/F$ and all finite subsets $B \subseteq H^i(E, \mu_\ell^{\otimes i - 1})$.
\end{defn}

The advantage of the generalized stable splitting dimension is that it provides information about higher degree cohomology groups as well, as in \cite[Corollary~1.4]{gosavi2019generalized}. 

\begin{prop}\label{inductive}
Let $F$ be a field of characteristic unequal to $\ell$.
  For all $i\geq j \ge 1$,
  $$\stspdim{i}{\ell}F\leq \gstspdim{j}{\ell}F$$ and
  $$\gstspdim{i}{\ell}F\leq \gstspdim{j}{\ell}F.$$
\end{prop}

\begin{proof}
Let $\alpha\in H^i(E, \mu_\ell^{\otimes i-1})$ for some finite extension $E$ of $F$ and $i\geq j$. 
By Remark~\ref{rescores}, we may assume that $E$ contains a primitive $\ell$-th root of
unity.  We can then use the norm residue isomorphism theorem as in the
proof of Lemma~\ref{l-power} in order to write $\alpha$ as a finite sum $\alpha=\sum_k\beta_{k}\cup \gamma_{k}$ where $\beta_{k}\in  H^j(E, \mu_\ell^{\otimes j-1}) = H^j(E, \mu_\ell^{\otimes j})$.
By definition, there exists a finite extension $L$ of $E$ such that the $\ell$-adic valuation of $[L:E]$ is at most $\gstspdim{j}{\ell}F$ and such that $L$ splits all $\beta_k$ occurring in the sum. But then $L$ splits $\alpha$, and the first claim follows.
Note that the same argument applies to finite collections of cohomology classes, hence the second assertion.
\end{proof}

The next lemma shows another useful property of the generalized stable splitting dimension. 

\begin{lem}\label{inductive_residue}
If $K$ is a complete discretely valued field having residue field $k$ with $\cha(k) \ne \ell$, then
\[\gstspdim{i}{\ell}{K} \le \gstspdim{j}{\ell}{k}\]
for all positive integers $i>j$.
\end{lem}

\begin{proof}
Since any finite extension of $K$ is of the same form, it suffices to consider classes defined over $K$. Let $\alpha_1,\ldots, \alpha_m\in H^i(K,\mu_\ell^{\otimes i-1})$. By the Witt decomposition theorem (\cite{GSz}, Corollary~6.8.8), that cohomology group is isomorphic to $H^i(k,\mu_\ell^{\otimes i-1})\oplus H^{i-1}(k,\mu_\ell^{\otimes i-2})$, so each $\alpha_r$ is of the form $(\beta_r, \beta_r')$, where $\beta_r, \beta_r'$ are classes over the residue field of degree $i$ and $i-1$, respectively. As in the proof of Proposition~\ref{inductive} above, we may assume that $K$ contains a primitive $\ell$-th root of unity and we may 
write $\beta_r$ and $\beta_r'$ as sums of
terms that are each of the form $\gamma \cup \delta$ where $\gamma \in H^j(k, \mu_\ell^{\otimes j-1})$.
But then all $\beta_r, \beta_r'$ are split by a 
finite extension $k'/k$ such that the $\ell$-adic valuation of $[k':k]$ is at most $\gstspdim{j}{\ell}{k}$. 
Since $K$ is complete, this extension lifts to a finite extension $K'/K$ of the same degree (by applying \cite[Th\'eor\`eme~I.6.1]{SGA1} to lift the maximal separable subextension, and then iteratively lifting $p$-th roots for the purely inseparable part).  This lift then splits $\alpha_1,\ldots, \alpha_m$, by the Witt decomposition theorem applied to $K'$ and $k'$.
\end{proof}

Our main result is the following theorem, which is proven in Section~\ref{main thm sect}.

\begin{thm} \label{main thm}
Suppose $k$ is a field and $\ell$ is a prime unequal to the characteristic of $k$. Let $k(x)$ denote the rational function field over $k$ in one variable.
Let $K$ be a complete discretely valued field with residue field~$k$, and let $F$ be a one-variable function field over~$K$. Then for all $i \ge 1$,
$$\stspdim{i}{\ell}{F} \leq \stspdim{i}{\ell}{ k} + \stspdim{ i}{\ell}{k(x)} +\begin{cases}2 &\text{if\,
 $\ell$ is odd}\\3&\text{if\, $\ell=2$}\end{cases}$$
and
$$\gstspdim{i}{\ell}{F} \leq \gstspdim{i}{\ell}{ k} + \gstspdim{ i}{\ell}{k(x)} +\begin{cases}2 &\text{if\, $\ell$ is odd}\\3&\text{if\, $\ell=2$.}\end{cases}$$
\end{thm}

The main interest is in the case $i > 1$. In fact, $\stspdim{1}{\ell}{F}=1$ and $\gstspdim{1}{\ell}{F} = \infty$ for any field $F$ for which $F^\times/(F^\times)^\ell$ is infinite (in particular, for $F$ as in the theorem). This is because $H^1(E,\mbb Z/\ell\mbb Z) = E^\times/E^{\times \ell}$ is then infinite for any finite extension $E/F$, and because a non-trivial $\mbb Z/\ell\mbb Z$-torsor over $E$ corresponds to a field extension that splits only over itself. 
For the same reason, a non-trivial class $\alpha \in H^1(E,\mbb Z/\ell\mbb Z)$ satisfies $\ind(\alpha)=\ell$.

Even for $i > 1$, we do not assert that these bounds are sharp. Nevertheless, in light of this theorem and \cite[Theorem~5.5]{HHK}, it is natural to investigate more precisely how these quantities grow.  In particular, one might ask whether $\stspdim{i}{\ell}{F}$ and $\gstspdim{i}{\ell}{F}$ are 
bounded above by $\dim(F)-i+1$ for certain naturally occurring fields $F$;
i.e., those obtainable from a prime field by passing iteratively to
finite generated field extensions of transcendence degree one over a
given field, and to henselian discretely valued fields with a given
field as residue field.  Here, $\dim(F)$ is defined inductively,
with the dimensions of $\mbb F_p$ and $\mbb Q$ set equal to $1$ and $2$,
and with the dimension increasing by one at each iterative step.  But
proving such an assertion seems a long way off.

\section{Preliminaries from Patching}

The proof of the main theorem will use the patching framework introduced in \cite{HH} and \cite{HHK}, which we now recall.

Let $K$ be a complete discretely valued field with residue field $k$, valuation ring $\mc O_K$, and uniformizer~$t$. Let $F$ be a semiglobal field over $K$; i.e., a one-variable function field over $K$. A {\em normal model of $F$} is an integral $\mc O_K$-scheme $\ms X$ with function field $F$ that is flat and projective over $\mc O_K$ of relative
dimension one, and that is normal as a scheme. If $\ms X$ is regular, we call it a {\em regular model}. Such a regular
model exists by the main theorem in \cite{Lipman} (see also \cite{stacks-project}, Theorem 0BGP).
Let $\mc P$ be a finite nonempty set of closed points of $\ms X$ that contains all the singular points of the reduced closed fiber $\ms X_k^{\text{red}}$. Let $\mc U$ be the collection of connected components of the complement $\ms X_k^{\text{red}}\smallsetminus \mc P$ .

For each $U\in \mc U$, we consider the ring $R_{U} \subset F$
consisting of the rational functions on~${\ms X}$ that are regular at all points of $U$.
The $t$-adic completion $\wh{R}_{U}$ of $R_{U}$ is an $I$-adically complete domain, where
$I$ is the radical  of the ideal generated by $t$ in $\wh{R}_{U}$.
The quotient
$\wh{R}_{U}/I $ equals $k[U]$, the ring of regular functions on the integral affine curve $U$.  We write $F_U$ for the field of fractions of $\wh R_U$. If $V\subseteq U$, then $\wh R_U\subseteq \wh R_V$ and $F_U\subseteq F_V$.

Also, for
a (not necessarily closed)
 point $P$ of ${\ms X}_k^{\operatorname{red}}$, we let~$F_{P}$ denote the field of fractions of the complete local ring $\wh R_P:=\wh{\mc O}_{\ms X,P}$ of $\ms X$ at $P$,  and we let $\kappa(P)$ denote its residue field. The
fields of the form $F_P$, $F_U$ for $P\in \mc P$, $U\in \mc U$ (and the rings  $\wh R_P, \wh R_U$, respectively) are called {\it patches} on $\ms X$.

For a closed point $P \in {\ms X}_k^{\operatorname{red}}$, we consider height one primes $\wp$ of the complete local ring $\wh R_P$ that contain the uniformizing parameter $t \in {\mc O}_K$.
For each such $\wp$, we let $R_\wp$ be the localization of $\wh R_P$ at $\wp$, and we let $\wh R_\wp$ be its $t$-adic (or equivalently, its $\wp$-adic) completion; this is a complete discrete valuation ring. We write $F_\wp$ for the fraction field of $\wh R_\wp$. If $P$ is on the closure of $U$, we call such a $\wp$ a {\it branch at} $P$ {\it on} $U$.
Let $\mc B$ denote the set of all branches at points $P\in \mc P$ (each of which lies on some $U\in \mc U$). The fields $F_\wp$ (resp., rings $\wh R_\wp$) are referred to as the {\it overlaps} of the corresponding patches $F_P, F_U$ (resp., $\wh R_P, \wh R_U$).  For a branch $\wp$ at $P$ on $U$,
there is an inclusion $F_{P} \subset F_\wp$ induced by the inclusion $\wh{R}_P\subset \wh{R}_\wp$, and also an inclusion $F_{U} \subset  F_\wp$ that is induced by the inclusion $\wh{R}_{U} \hookrightarrow \wh{R}_\wp$. (See \cite{admis}, beginning of Section~4.)

The strategy for proving Theorem~\ref{main thm} relies on putting ourselves in the above context. Given a class $\alpha \in H^i(F, \mu_\ell^{\otimes i - 1})$, we will choose a suitable regular model $\ms X$, along with
$\mc P$ and $\mc U$, and will
construct splitting fields $L_\xi/F_\xi$ for $\alpha_{F_\xi}$, for each $\xi\in {\mc P}\cup \mc U$. Next, we will use these to obtain an extension $L/F$ that splits $\alpha$ locally.  Finally, 
we will use a local-global principle from \cite{Hi} to show that this extension in fact splits $\alpha$.  To handle the second of those three steps, we prove some auxiliary results, starting with a general lemma.  

\begin{lem}
\label{branch to point}
Let $v_1,\dots,v_n$ be distinct non-trivial discrete valuations on a field $E$, with completions $E_i$.  Let $d$ be a positive integer and for each $i$ let $L_i$ be an \'etale $E_i$-algebra of degree $d$.  Then there exists an \'etale $E$-algebra $L$ of degree $d$ such that $L \otimes_E E_i \cong L_i$ for all $i$.  If some $L_i$ is a field, then so is $L$.
\end{lem}

\begin{proof}
The complete discretely valued field $E_i$ is infinite for each $i$, and so by Corollary~4.2(d) of \cite{MR3590278} there is a primitive element for the \'etale algebra $L_i$ over $E_i$, say with monic minimal polynomial $f_i(x) \in E_i[x]$ of degree $d$.  For each $i$, there is an extension of $v_i$ to a discrete valuation on the polynomial ring $E_i[x]$, by taking the minimum of the valuations on the coefficients; we again write $v_i$ for that extension.
Applying Krasner's Lemma (\cite{Lang}, Prop. II.2.4) to each monic irreducible factor $f_{ij}$ of $f_i \in E_i[x]$ (and then taking the maximum) gives a positive integer $n_i$ such that for any monic polynomial $h_i\in E_i[x]$ of the same degree as some $f_{ij}$, if $v_i(h_i-f_{ij})>n_i$ then $h_i$ is irreducible, and the polyomials $h_i$ and $f_{ij}$ define the same field extension of $E_i$.  By a general form of Hensel's Lemma (see Theorem 8 of \cite{Brink}), for each $i$ there is an integer $m_i$ such that for any monic polynomial $g_i \in E_i[x]$ of degree equal to that of $f_i$ and with $v_i(f_i-g_i) > m_i$, we may write $g_i$ as a product of monic polynomials $g_{ij} \in E_i[x]$ of the same respective degrees as $f_{ij}$ and such that
$v_i(f_{ij}-g_{ij})>n_i$ for all $j$.

The field $E$ is dense in $\prod E_i$ by Theorem VI.7.2.1 of \cite{Bourbaki}.
Hence we may find a monic polynomial $f\in E[x]$ of degree $d$ such that $v_i(f_i-f)>m_i$ for all $i$.  By
the definition of $m_i$, we may write $f$ as a product of monic factors $g_{ij}$ over $E_i$ that are respectively of the same degrees as the polynomials $f_{ij}$ and with $v_i(f_{ij}-g_{ij})>n_i$.
By the definition of $n_i$, each factor $g_{ij}$ of $f$ over $E_i$ is irreducible and defines
the same field extension of $E_i$ as $f_{ij}$;
and so the \'etale algebras induced by $f$ and by $f_i$ over $E_i$ are the same.  Hence the
 \'etale $E$-algebra $L$ defined by $f$ induces $L_i$ over $E_i$ for all $i$. The last assertion is clear.
\end{proof}

Resuming our notation for semiglobal fields, we have the following.

\begin{lem}\label{U extension}
Given $F$, $\mc X$, and $\mc U$ as above, suppose that for each $U\in \mc U$ we are given an \'etale $F_U$-algebra $L_U$ of (a common) degree~$d$. Then there exists an \'etale $F$-algebra $L$ (necessarily of degree~$d$) such that $L \otimes_F F_U \cong L_U$ for all $U$. If some $L_U$ is a field, so is $L$.
\end{lem}

\begin{proof}
For a point $P \in \mc P$, each branch $\wp$ at $P$ lies on the closure of a unique $U\in \mc U$; and we
define an \'etale $F_{\wp}$-algebra $L_{\wp}:=L_U\otimes_{F_U} F_{\wp}$. 
Applying Lemma~\ref{branch to point} to the field $F_P$ and the discrete valuations corresponding to the branches at $P$, we obtain an \'etale $F_P$-algebra $L_P$ such that $L_P\otimes F_{\wp}\cong L_{\wp}$ for each of the branches $\wp$ at $P$. Therefore, we have defined a system of \'etale $F_\xi$-algebras $L_\xi$ for $\xi\in {\mc P}\cup {\mc U}$, together with isomorphisms $L_P\otimes_{F_P}F_{\wp}\cong L_U\otimes_{F_U}F_{\wp}$ whenever $\wp$ is a branch at $P$ on $U$. Since patching holds for \'etale algebras in this context (see, for example, Proposition~3.7 and Example~2.7 in \cite{HHK:pop}), there is an \'etale $F$-algebra $L$ with the desired properties. The final assertion is clear.
\end{proof}

The next lemma is a variant of \cite[Theorem~2.6]{HHKPS:Gal}.

\begin{lem} \label{point extension}
With $K$, $k$, $F$, and $\mc P$ as above, suppose that for each $P\in \mc P$, we are given an \'etale $F_P$-algebra $L_P$ of (a common) degree~$d$ prime to the characteristic of $k$, and assume that the integral closure of $\wh R_P$ in $L_P$ is unramified over $\wh R_P$. Then there exists an \'etale $F$-algebra $L$ such that $L \otimes_F F_P \cong L_P$ for all $P$. If some $L_P$ is a field, so is $L$.
\end{lem}

\begin{proof}
For each $P \in \mc P$, the normalization $S_P$ of $\wh R_P$ in $L_P$ is a degree $d$ \'etale $\wh R_P$-algebra. It induces an \'etale $\wh R_\wp$-algebra $S_\wp$ for each branch $\wp$ at $P$; and those in turn induce degree $d$ \'etale 
algebras $L_\wp$ over $F_\wp$ and $\lambda_\wp$ over $\kappa(\wp)$, where $\kappa(\wp)$ is the residue field at $\wp$.  The branches $\wp$ on $U$ define distinct non-trivial discrete valuations on the function field $k(U)$ of $U$, with completions $\kappa(\wp)$.  Applying Lemma~\ref{branch to point}, we obtain a degree $d$ \'etale algebra $\Lambda_U$ over $k(U)$ such that $\Lambda_U \otimes_{k(U)} \kappa(\wp) \cong \lambda_\wp$ for all branches $\wp$ on $U$.  The normalization of $k[U]$ in $\Lambda_U$ is a generically \'etale $k[U]$-algebra $\bar A_U$ that induces $\Lambda_U$ over $k(U)$. 
By lifting the defining coefficients of $\bar A_U$ from $k[U]$ to $\wh R_U$, we obtain a generically \'etale $\wh R_U$-algebra $A_U$ whose reduction is $\bar A_U$.   
The algebra $A_U$ induces $S_\wp$ over $\wh R_\wp$, because both 
$A_U \otimes_{\wh R_U} \wh R_\wp$ and $S_\wp$
lift the \'etale $\kappa(\wp)$-algebra $\lambda_\wp$, and that lift is unique by \cite[Th\'eor\`eme~I.5.5]{SGA1}.  Thus $L_U := A_U \otimes_{\wh R_U} F_U$ is an \'etale algebra over $F_U$ that induces $L_\wp := S_\wp \otimes_{\wh R_\wp} F_\wp$ over $F_\wp$.  

Thus we have \'etale algebras $L_P$ over $F_P$ for each $P \in \mc P$ and $L_U$ over $F_U$ for each $U \in \mc U$ such that $L_P$ and $L_U$ induce the same $F_\wp$-algebra $L_\wp$ for $\wp$ a branch at $P$ on $U$.  
By the patching result \cite[Theorem~7.1(iii)]{HH} (in the context of \cite[Theorem~6.4]{HH} and \cite[Proposition~3.3]{HHK:H1}), 
there is an \'etale algebra $L$ over $F$ that induces $L_U$ over $F_U$ for all $U \in \mc U$ and induces $L_P$ over $F_P$ for all $P \in \mc P$.
This yields the main assertion, and the 
final assertion of the lemma is clear.
\end{proof}

\section{Splitting unramified cohomology classes}

In order to prove the main theorem, we will reduce to the case of unramified classes. Let~$L$ be a field. For every discrete valuation $v$ of~$L$, we let $\kappa(v)$ denote its residue field. Recall that for a prime $\ell\neq \operatorname{char}(\kappa(v))$
and $i\ge 1$, 
there is a residue homomorphism $\res_v: H^i(L, \mu_\ell^{\otimes i - 1})\rightarrow H^{i-1}(\kappa(v), \mu_\ell^{\otimes i - 2})$; 
e.g., see \cite[Section~II.7.9, p.~18]{GMS}. A class $\alpha \in H^i(L, \mu_\ell^{\otimes i - 1})$ is called {\em unramified at $v$} if $\res_v(\alpha)=0$. If $\ms Y$ is a regular integral scheme with function field $L$ and $\ms Y^{(1)}$ is the set of codimension one points of~$\ms Y$, then every $y\in \ms Y^{(1)}$ defines a discrete valuation $v_y$ of $L$. We say that $\alpha$ as above is {\em unramified at $y$} if $\res_{v_y}(\alpha)=0$. It is {\em unramified on $\ms Y$} if it is unramified at all points of $\ms Y^{(1)}$; and we write $H^i(L,\mu_\ell^{\otimes i - 1})^{nr,\ms Y}$ for the subgroup of $H^i(L,\mu_\ell^{\otimes i - 1})$ consisting of these unramified classes. 

\begin{lem}\label{U_unramified}
With notation as above and $U\in \mc U$, let $\alpha \in H^i(F_U, \mu_\ell^{\otimes i - 1})$ be unramified on $\operatorname{Spec}(\wh R_U)$.
Then for some nonempty affine open subset $U'\subseteq U$, $\alpha_{F_{U'}}$ is in the image of $H^i(\wh R_{U'}, \mu_\ell^{\otimes i - 1})\rightarrow H^i(F_{U'}, \mu_\ell^{\otimes i - 1})$.
\end{lem}

\begin{proof}
Let $R_\eta^h:=\varinjlim_{V\subseteq U} \wh R_{V}$ (varying over the nonempty open subsets $V\subseteq U$), and let $F_\eta^h$ be its fraction field. Then by \cite{Hi}, Lemma~3.2.1, $R_\eta^h$ is a henselian discrete valuation ring with residue field $k(U)$, and $F_\eta^h=\varinjlim_{V\subseteq U} F_{V}$. Since  $\alpha$ is unramified, so is its image $\alpha_{F_\eta^h}$. Thus by \cite{CTSB}, beginning of Section~3.3, $\alpha_{F_\eta^h}$ is the image of some $\til{\alpha}\in H^i(R_\eta^h, \mu_\ell^{\otimes i-1})$. According to \cite{stacks-project},  Theorem~09YQ,  $H^i(R_\eta^h, \mu_\ell^{\otimes i-1})=\varinjlim_{V\subseteq U} H^i(\wh R_V, \mu_\ell^{\otimes i-1})$ and $H^i(F_\eta^h, \mu_\ell^{\otimes i-1})=\varinjlim_{V\subseteq U} H^i(F_V, \mu_\ell^{\otimes i-1})$. In particular, there is some nonempty open  subset $V\subseteq U$ so that $\til{\alpha}$ is the image of an element $\til{\alpha}'\in H^i(\wh R_V,\mu_\ell^{\otimes i-1})$. The classes $\alpha_{F_V}$ and $\til{\alpha}'_{F_V}$ then have the same image in $H^i(F_\eta^h,\mu_\ell^{\otimes i-1})$ by construction. Again by Lemma~3.2.1 of \cite{Hi}, $F_\eta^h=\varinjlim_{W\subseteq V} F_W$, and thus there exists a $U'\subseteq V$ for which $\alpha_{F_{U'}}=\til{\alpha}'_{F_{U'}}$. But then $U'$ is as desired.
\end{proof}

This next result gives a bound on the index in the case of unramified cohomology classes.

\begin{prop}\label{unramified}
Let $K$ be a complete discretely valued field with residue field~$k$, let $\ell$ be a prime unequal to the characteristic of $k$, let $F$ be the function field of a $K$-curve, and   
let $\ms X$ be a regular model of $F$. Let $i \ge 1$.
\renewcommand{\theenumi}{\alph{enumi}}
\renewcommand{\labelenumi}{(\alph{enumi})}
\begin{enumerate}
\item \label{one} If $\alpha\in H^i(F, \mu_\ell^{\otimes i - 1})$ is unramified on $\ms X$, then
$$\ind(\alpha)\mid \ell^{\stspdim{i}{\ell}{k}+ \stspdim{i}{\ell}{k(x)}}.$$
\item \label{two} If $B\subseteq H^i(F, \mu_\ell^{\otimes i - 1})$ is a finite collection of cohomology classes that are unramified on $\ms X$,
then $$\ind(B)\mid \ell^{\gstspdim{i}{\ell}{k}+ \gstspdim{i}{\ell}{k(x)}}.$$
\end{enumerate}
\end{prop}

\begin{proof}
Both assertions are trivially true for $i=1$, by the paragraph following Theorem~\ref{main thm}. So we assume $i>1$ from now on. 

We start by proving part~(\ref{two}).
Let $B=\{\alpha_j \,|\, j\in J\}$ for some finite index set $J$.
By Lemma~\ref{l-power}, it is sufficient to show that there is a finite field extension $L/F$ that splits all classes in~$B$ and such that the $\ell$-adic valuation of $[L:F]$ is at most $\gstspdim{i}{\ell}{k}+\gstspdim{i}{\ell}{k(x)}$. 
Let $\mc P$ be a finite nonempty subset of the closed fiber containing all the singular points of $\ms X_k^{\operatorname{red}}$, and let $\mc U$ be the set of components of the complement $\ms X_k^{\operatorname{red}}\smallsetminus \mc P$.

Fix $U \in \mc U$. After deleting finitely many points from $U$ and adding those to $\mc P$, we may assume that each $(\alpha_j)_{F_{U}}$ is the image of some $\til{\alpha}_j\in H^i(\wh R_{U}, \mu_\ell^{\otimes i - 1})$, by Lemma~\ref{U_unramified}. This gives
$$H^i(\wh R_{U}, \mu_\ell^{\otimes i - 1})\cong H^i(U, \mu_\ell^{\otimes i - 1})\rightarrow H^i(k(U), \mu_\ell^{\otimes i - 1}),\quad \til{\alpha}_j\mapsto \bar\alpha_j,$$
where the isomorphism is by Gabber's affine analog of proper base change (\cite{stacks-project}, Theorem~09ZI). By definition of the generalized stable splitting dimension, there exists a finite field extension $l_U$ of $k(U)$ that splits all $\bar\alpha_j$ and so that the $\ell$-adic valuation of $[l_U:k(U)]$ is at most
$\gstspdim{i}{\ell}{k(x)}$. Let $l_U'$ be the separable closure of $k(U)$ in $l_U$. Then since $[l_U:l_U']$ is a power of $\operatorname{char}(k)$ and thus prime to $\ell$, the separable extension $l_U'$ also splits all $\bar\alpha_j$ (see Remark~\ref{rescores}). Let  $V\rightarrow U$ be the normalization of $U$ in $l_U'$, so that  $l_U'=k(V)$. Hence each $\til{\alpha}_j$ maps to zero under the composition
$$H^i(\wh R_{U}, \mu_\ell^{\otimes i - 1})\cong H^i(U, \mu_\ell^{\otimes i - 1})\rightarrow H^i(k(U), \mu_\ell^{\otimes i - 1})\rightarrow H^i(k(V), \mu_\ell^{\otimes i - 1}).$$

The collection of $V\times_U U'$, where $U'$ ranges over the non-empty open subsets of $U$, is cofinal in the collection of non-empty open subsets $V'\subseteq V$.  So by \cite[Theorem~09YQ]{stacks-project}, 
\[H^i(k(V), \mu_\ell^{\otimes i - 1})=\varinjlim_{V'\subseteq V} H^i(V',  \mu_\ell^{\otimes i - 1}) =\varinjlim_{U'\subseteq U} H^i(V\times_UU',\mu_\ell^{\otimes i - 1}).\]
Hence there exists some $U'\subseteq U$ for which each $\til{\alpha}_j$ maps to zero in $H^i(V\times_UU',\mu_\ell^{\otimes i - 1})$. Since $k(V)/k(U)$ is separable, $V \to U$ is generically \'etale. Possibly after shrinking $U'$, we may assume that $V\times_U U'\rightarrow U'$ is finite \'etale. Let $I$ be the ideal defining $U'$ in $\operatorname{Spec}(\wh R_{U'})$. Then $(\wh R_{U'},I)$ is a henselian pair, so  $V\times_U U'\rightarrow U'$ is the closed fiber of a finite  \'etale cover $\operatorname{Spec}(S_{U'})\rightarrow \operatorname{Spec}(\wh R_{U'})$ of the same degree by \cite{stacks-project}, Lemma 09XI. Note that $\operatorname{Spec}(S_{U'})$ is reduced and irreducible since $V$ is, and hence $S_{U'}$ is an integral domain. The commutative diagram
\[\begin{tikzcd}
H^i(\wh R_U, \mu_\ell^{\otimes i - 1})\arrow{r}&H^i(\wh R_{U'}, \mu_\ell^{\otimes i - 1})\arrow{r}{\cong}\arrow{d}
&H^i(U', \mu_\ell^{\otimes i - 1})\arrow{d}\\
&H^i(S_{U'}, \mu_\ell^{\otimes i - 1})
\arrow{r}{\cong}
&H^i(V\times_U U', \mu_\ell^{\otimes i - 1})
\end{tikzcd}\]
then shows that each $\til{\alpha}_j$ maps to zero in $H^i(S_{U'}, \mu_\ell^{\otimes i - 1})$; hence all $\alpha_j$ are split by the fraction field $E_{U'}$ of $S_{U'}$, which is an extension of $F_{U'}$ whose degree has $\ell$-adic valuation at most $\gstspdim{i}{\ell}{k(x)}$.
(Here the isomorphisms in the diagram are -- again -- by Gabber's affine analog of proper base change, \cite[Theorem~09ZI]{stacks-project}.) Note that each $U'$ was obtained by removing a finite number of closed points from the corresponding $U\in \mc U$. We add those points to $\mc P$ and replace $\mc U$ with the set of components of the complement in $\ms X_k^{\operatorname{red}}$ of this possibly enlarged set $\mc P$ (the elements of this new set $\mc U$ are exactly the sets $U'$). Let $d_1$ be the least common multiple of the degrees $[E_{U'}:F_{U'}]$ 
where $U'$ is in the (new) set $\mc U$.  
Thus the $\ell$-adic valuation of $d_1$ is at most $\gstspdim{i}{\ell}{k(x)}$.
By taking direct sums of an appropriate number of copies of $E_{U'}$ for each such $U'$, we obtain \'etale $F_{U'}$-algebras $L_{U'}$ for all $U'$ of degree $d_1$.
Then by Lemma~\ref{U extension}, there is an \'etale $F$-algebra $L_1$ of degree~$d_1$ so that $L_1\otimes_F F_{U'}\cong L_{U'}$ for all $U'\in \mc U$.

For $P\in \mc P$, each class $\alpha_{j,P}:=(\alpha_j)_{F_P}$ is unramified on $\operatorname{Spec}(\wh R_P)$, since each $\alpha_j$ is unramified. Thus by \cite{Sakagaito}, Theorem~9, we may lift each $\alpha_{j,P}$ to a class in $H^i(\wh R_P, \mu_\ell^{\otimes i - 1})$; that group is isomorphic to $H^i(\kappa(P),\mu_\ell^{\otimes i - 1})$ by proper base change (\cite{SGA4}, Exp. XII, Corollaire~5.5). By definition of the generalized stable splitting dimension, we may find a common splitting field $l_P/\kappa(P)$ for the images of the $\alpha_{j,P}$, so that $[l_P: \kappa(P)]$ has  $\ell$-adic valuation at most $\gstspdim{i}{\ell}{k}$. As in the previous part, we may assume that $l_P/\kappa(P)$ is separable. By \cite{SGA1}, Theorem~I.6.1, the extension lifts to a finite \'etale $\wh R_P$-algebra $S_P$ of the same degree (using the completeness of $\wh R_P$). Note that again by proper base change (loc. cit.), all $\alpha_{j,P}$ split over $S_P$. Since $\wh R_P$ is a regular local domain,
and since $S_P$ is finite \'etale over $\wh R_P$ and lifts $l_P$, $S_P$ is a regular local domain. Its fraction field is a finite extension $E_P/F_P$  of the same degree, which splits all $\alpha_{j,P}$. Let $d_2$ be the least common multiple of the degrees $[L_P:F_P]$. By taking direct sums of an appropriate number of copies of $E_P$ for each $P\in \mc P$, we obtain \'etale $F_P$-algebras $L_P$ (for all $P$) of degree $d_2$ which has $\ell$-adic valuation at most $\gstspdim{i}{\ell}{k}$. Then by Lemma~\ref{point extension}, there is an \'etale $F$-algebra $L_2$ of degree~$d_2$ so that $L_2\otimes_F F_P\cong L_P$ for all $P\in \mc P$.

Consider the tensor product $L_1\otimes_F L_2$; this is a direct sum of finite field extensions of $F$ since each $L_i$ is an \'etale $F$-algebra. Since the $\ell$-adic valuation of the degree of $L_1\otimes_F L_2$ is at most $\gstspdim{i}{\ell}{k(x)}+ \gstspdim{i}{\ell}{k}$, the same is true for at least one of the direct summands, say~$L/F$.
Let $\ms X_L$ be the normalization of $\ms X$ in $L$, let $\mc P_L$ be the preimage of $\mc P$ under the natural map $\ms X_L\rightarrow \ms X$, and let $\mc U_L$ be the set of connected components of the complement of $\mc P_L$ in the reduced closed fiber of $\ms X_L$.
For each $P\in \mc P$, $L\otimes_F F_P$ is the direct product of the fields $L_{P'}$, where $P'$ runs over the points of $\mc P_L$ that map to $P$  and $L_{P'}$ is the fraction field of the complete local ring of $\ms X_L$ at $P'$; similarly for each $U\in \mc U$. Hence all $(\alpha_j)_{L_\xi}$ are split for every $\xi \in \mc P_L \cup \mc U_L$. By Theorem~3.1.5 of \cite{Hi}, all
$\alpha_j$ are split over $L$.  This completes the proof of part~(\ref{two}).

For part~(\ref{one}), note that if $\alpha$ is a single class unramified on a regular model $\ms X$, 
then for splitness over each $U\in \mc U$ (resp.\ $P\in \mc P$), it suffices to take an extension whose degree has $\ell$-adic valuation at most $\stspdim{i}{\ell}{k(x)}$ (resp.\ $\stspdim{i}{\ell}{k}$), by definition of the stable splitting dimension. Hence the above proof yields a splitting field $L$ for $\alpha$ whose degree over $F$ has $\ell$-adic valuation at most $\stspdim{i}{\ell}{k}+\stspdim{i}{\ell}{k(x)}$.
Since $\ind(\alpha)$ is an $\ell$-power by Lemma~\ref{l-power}, this implies
$$\ind(\alpha)\mid \ell^{\stspdim{i}{\ell}{k}+ \stspdim{i}{\ell}{k(x)}}$$
as we intended to show.
\end{proof}

\begin{rem}
If $k$ is finite in the context of Proposition~\ref{unramified}, then the group $H^i(F, \mu_\ell^{\otimes i-1})^{nr,\ms X}$ vanishes for all $i>1$.  This follows from \cite[Th\'eor\`eme~III.3.1, Corollaire~II.1.10]{GpBr} for $i=2$; from \cite[Proposition~5.2]{Kato} for $i=3$; and because $\cd(F)=3$ for $i \ge 4$.  Thus the proposition applies only to the zero class in these situations, and so it has no actual content there.  (In the case of $i=1$, as noted at the beginning of the above proof, the assertion of Proposition~\ref{unramified} is trivial for an arbitrary residue field $k$.)
\end{rem}

\section{Proof of the main theorem} \label{main thm sect}

We are now in a position to prove the main theorem.

\begin{proof}[Proof of Theorem~\ref{main thm}]
We first prove the second assertion. Let $B \subseteq H^i(F, \mu_\ell^{\otimes i - 1})$ be a finite collection of cohomology classes, and choose a regular model $\ms X$ of $F$.  By \cite{gosavi2019generalized}, Prop.~3.1, there is a field extension $L/F$ of degree $\ell^2$ (resp. $2^3=8$) for $\ell$ odd (resp. $\ell=2$) that splits the ramification of $B$ with respect to all discrete valuations on $L$ whose restriction to $F$ has a center on $\ms X$.  
The extension $L/F$ corresponds to a morphism $\ms Y \to \ms X$ for some regular model $\ms Y$ of $L$; and  $\alpha_L \in H^i(L, \mu_\ell^{\otimes i-1})^{nr,\ms Y}$ for every $\alpha \in B$.
By Proposition~\ref{unramified}(b), there exists a finite field extension $\til L/L$  that splits all elements of~$B$ and so that $[\til{L}:L]$ has $\ell$-adic valuation at most $ \gstspdim{i}{\ell}{k}+\gstspdim{i}{\ell}{k(x)}$. Thus the $\ell$-adic valuation of $[\til L:F]$ is at most
$$\gstspdim{i}{\ell}{k}+ \gstspdim{i}{\ell}{k(x)}+\begin{cases}2 &\text{if $\ell$ is odd}\\3&\text{if $\ell=2$.}\end{cases}$$

To bound the generalized stable splitting dimension, we also need to consider cohomology classes defined over finite field extensions $E/F$. Each such $E$ is the function field of a curve over $K_E$, where $K_E$ is some finite extension of~$K$ and hence is a complete discretely valued field whose residue field $k'$ is a finite extension of $k$. Now if  $B \subseteq H^i(E, \mu_\ell^{\otimes i - 1})$ is a finite collection of cohomology classes, the first part of the proof shows the existence of a common splitting field $L/E$ for the elements of $B$ whose degree $[L:E]$ has $\ell$-adic valuation at most
\begin{align*}&\gstspdim{i}{\ell}{k'}+ \gstspdim{i}{\ell}{k'(x)}+\begin{cases}2 &\text{if $\ell$ is odd}\\3&\text{if $\ell=2$}\end{cases}\\
&\leq \gstspdim{i}{\ell}{k}+ \gstspdim{i}{\ell}{k(x)}+\begin{cases}2 &\text{if $\ell$ is odd}\\3&\text{if $\ell=2$,}\end{cases}\end{align*}
which proves the desired bound for $\gstspdim{i}{\ell}{F}$.

If $B=\{\alpha\}$ is a one element set, Proposition~\ref{unramified}(a) gives
$\ind(\alpha_L)\mid \ell^{\stspdim{i}{\ell}{k}+ \stspdim{i}{\ell}{k(x)}}$, and hence
$\ind(\alpha)\mid \ell^{m}$ where $$m=\stspdim{i}{\ell}{k}+ \stspdim{i}{\ell}{k(x)}+\begin{cases}2 &\text{if $\ell$ is odd}\\3&\text{if $\ell=2$.}\end{cases}$$ Since $\alpha$ was arbitrary, this shows that
$$\spdim{i}{\ell}{F}\leq \stspdim{i}{\ell}{k}+ \stspdim{i}{\ell}{k(x)}+\begin{cases}2 &\text{if $\ell$ is odd}\\3&\text{if $\ell=2$.}\end{cases}$$
As before, the same bound applies to finite extensions $E/F$, and hence
$$\stspdim{i}{\ell}{F}\leq \stspdim{i}{\ell}{k}+ \stspdim{i}{\ell}{k(x)}+\begin{cases}2 &\text{if $\ell$ is odd}\\3&\text{if $\ell=2$,}\end{cases}$$
as we wanted to show.
\end{proof}

\section{Bounds for higher rank complete discretely valued fields}

In this section, we bound $\gstspdim{i}{\ell}{F}$ for one-variable function fields $F$ over higher rank complete discretely valued fields -- that is, fields $k_r$ arising in an iterated construction of fields $k_0, k_1, \ldots, k_r$ where $k_j$ is a complete discretely valued field with residue field $k_{j-1}$, for all $j\geq 1$. We will do this using Theorem~\ref{main thm}. We first determine the generalized stable splitting dimension of higher rank complete discretely valued fields.

\begin{lem} \label{discrete index bootstrap}
Let $k$ be a field and let $\ell\neq \cha(k)$ be a prime.
Let $r \ge 0$, and let $k_0, k_1, \ldots, k_r$ be a sequence of fields with $k_0 = k$, and $k_j$ a complete discretely valued field with residue field $k_{j-1}$ for all $j \ge 1$. 
Then for every finite collection $B\subseteq H^i(k_r,\mu_\ell^{\otimes i-1})$, there exists an extension $L/k_r$ of degree dividing $\ell^{\gstspdim{i}{\ell}{k}+r}$ that splits all elements of~$B$. In particular, 
$\gstspdim{i}{\ell}{k_r} \le \gstspdim{i}{\ell}{k} + r$.
The same statements remain true when $B$ is replaced by a single class and $\gstspdim{i}{\ell}{-}$ is replaced with~$\stspdim{i}{\ell}{-}$. 
\end{lem}

\begin{proof}
By induction, it suffices to prove the result with $r=1$. Set $K = k_1$, let $v$ denote the valuation on $K$, and 
let $A$ be its valuation ring, with uniformizer $\pi$.  
By proper base change (\cite{SGA4}, Exp. XII, Corollaire~5.5), for any $m \ge 1$ the mod $\pi$ reduction map 
$H^m(A, \mu_\ell^{\otimes m-1}) \to H^m(k, \mu_\ell^{\otimes m-1})$ is an isomorphism, and so we may identify these two cohomology groups.  Thus
by \cite[Proposition~II.7.11, p.~18]{GMS}, each element $\alpha \in H^i(K, \mu_\ell^{\otimes i - 1})$ may be written in the form $\alpha' + (\pi) \cup \beta$, where $\alpha' \in H^i(A, \mu_\ell^{\otimes i - 1})$; where
$(\pi) \in H^1(K,\mu_\ell)$ is the class defined by $\pi$; and where $\beta \in H^{i-1}(A, \mu_\ell^{\otimes i - 2})$ is the class identified with $\res_v(\alpha) \in H^{i-1}(k, \mu_\ell^{\otimes i - 2})$ via the above isomorphism.
Consequently, if we base change to $\til K = K(\sqrt[\ell]{\pi})$ to split the class $(\pi)$, we find that $(\alpha)_{\til K} = (\alpha')_{\til K}$.

Now let $B=\{\alpha_1,\ldots,\alpha_m\} \subseteq H^i(K, \mu_\ell^{\otimes i - 1})$ be a finite collection, and let $\ov{B} = \{\ov{\alpha'_1}, \ldots, \ov{\alpha'_m}\}$, where $\ov{\alpha_i' }$ denotes the image of $\alpha_i'$ in $H^i(k, \mu_\ell^{\otimes i - 1})$ (and $\alpha_i'$ is associated to $\alpha_i$ as in the first part of the proof). By definition, there exists a splitting field $k'/k$ for $\ov{B}$ of degree dividing $\ell^{\gstspdim{i}{\ell}{k}}$. To prove the first assertion of the lemma, it suffices to show that we may find a splitting field $\til K'/K$ of $B$ whose degree divides $\ell[k':k]$.
By hypothesis on the characteristic, each $\ov{\alpha'_i}$ is also split by the separable closure of $k$ in $k'$ (Remark~\ref{rescores}), and so we may assume without loss of generality that $k'$ is a separable extension of $k$. Consequently, we may lift $k'$ to an unramified extension $A'$ of $A$ of the same degree; let  $K'$ denote the fraction field of $A'$.  Again using proper base change (\cite{SGA4}, Exp.~XII, Corollaire~5.5), the classes $(\alpha'_i)_{A'}$ are split; so it
follows that the classes $(\alpha'_i)_{K'}$ are split as well. Let $\til K'$ be a compositum of $\til K$ and $K'$. Then $(\alpha_i)_{\til K}=(\alpha'_i)_{\til K}=0$. As $[\til K': K] \,|\, \ell [k':k]$, the extension $\til K'/K$ is as desired. The assertion on the generalized stable splitting dimension is an immediate consequence. 

If $B$ consists of a single class, then the extension $k'/k$ in the previous part can be chosen of degree dividing $\ell^{\stspdim{i}{\ell}{k}}$, and this yields the final assertion of the lemma.
\end{proof}

\begin{rem}
The bounds given in the previous lemma are not sharp. For example, consider $k={\mathbb Q}$ and $i=2=\ell$. Given a collection of $2$-torsion Brauer classes, we may find a quadratic extension of ${\mathbb Q}$ which is non-split at every prime where at least one of the corresponding quaternion algebras is ramified. This extension will then split all the classes, so $\gstspdim{2}{2}{{\mathbb Q}}=1$, and 
$\gstspdim{3}{2}{{\mathbb Q}}\leq\gstspdim{2}{2}{{\mathbb Q}}=1$ by Proposition~\ref{inductive}. Since the Pfister form $\langle\langle -1,-1,-1\rangle\rangle$ does not split over ${\mathbb Q}$, $\gstspdim{3}{2}{{\mathbb Q}}=1$.
Lemma~\ref{discrete index bootstrap} then gives $\gstspdim{3}{2}{{\mathbb Q}((t)))}\leq 2$. But more is true: 
since $\gstspdim{2}{2}{{\mathbb Q}}=1$, Lemma~\ref{inductive_residue} implies the stronger assertion that $\gstspdim{3}{2}{{\mathbb Q}((t)))}=1$ (note that the above Pfister form does not split over ${\mathbb Q}((t))$ either).
\end{rem}

\begin{thm} \label{new iterative}
Let $k$ be a field, let $\ell\neq \cha(k)$ be a prime, let
$d = \gstspdim{i}{\ell}{k}$, and let $\delta = \gstspdim{i}{\ell}{k(x)}$.
Suppose we are given a sequence $k = k_0, k_1, \ldots, k_r$ of fields with $k_j$ a complete discretely valued field having residue field $k_{j-1}$ for all $j \ge 1$.
Then
\[\gstspdim{i}{\ell}{F} \leq \begin{cases} \delta + \frac r 2 (r + 2d + 3) &\text{if $\ell$ is odd}\\
\delta + \frac r 2 (r + 2d + 5) & \text{if $\ell=2$}\end{cases}\]
for any one variable function field $F$ over $k_r$.
The same result holds for $\stspdim{i}{\ell}{F}$ when $d$ and $\delta$ are replaced with $\stspdim{i}{\ell}{k}$ and $\stspdim{i}{\ell}{k(x)}$, respectively.
\end{thm}

\begin{proof}
Note that by definition of the invariants in question, it suffices to consider the case $F=k_r(x)$. By Lemma~\ref{discrete index bootstrap}, we know that $\gstspdim{i}{\ell}{k_j} \leq \gstspdim{i}{\ell}{k} + j= d + j$. Let $\varepsilon$ be $2$ if $\ell$ is odd and let it be $3$ if $\ell$ is even. By Theorem~\ref{main thm}, we have
$\gstspdim{i}{\ell}{k_{j}(x)} \leq \gstspdim{i}{\ell}{k_{j-1}} + \gstspdim{i}{\ell}{k_{j-1}(x)} + \varepsilon$, and so
\[\gstspdim{i}{\ell}{k_{j}(x)} - \gstspdim{i}{\ell}{k_{j-1}(x)} \leq d + j - 1 + \varepsilon.\]
Taking a sum of these inequalities for $j = 1, \ldots, r$ yields
\[\gstspdim{i}{\ell}{k_{r}(x)} - \gstspdim{i}{\ell}{k_{0}(x)} \leq rd + \frac{r(r-1)}{2} + r\varepsilon\]
and so
\[\gstspdim{i}{\ell}{k_{r}(x)} \leq rd + \frac{r(r-1)}{2} + \delta + r\varepsilon = \delta + \frac r 2 (r + 2d + 2\varepsilon - 1),\]
as desired. The proof for the stable splitting dimension is similar (using the corresponding assertions of Lemma~\ref{discrete index bootstrap} and Theorem~\ref{main thm}).
\end{proof}

Next, we would like to examine the behavior of the splitting dimension as the cohomological degree varies. While we don't have the ability to control this well for general fields, we can make some statements to this effect in the case that the cohomological dimension is bounded, using that $\gstspdim{m}{\ell}{k}=0$ for $m>\cd_\ell(k)$.

\begin{thm} \label{cd iterative sgf} 
Let $k$ be a field, let $\ell\neq \cha(k)$ be a prime, and let
$c = \cd_{\ell}(k)$.
Consider a sequence of fields $k=k_0, k_1, \ldots, k_r$ where $k_j$ is a complete discretely valued field having residue field $k_{j-1}$ for all $j \ge 1$.
Set $\varepsilon = 2$ if $\ell$ is odd and $\varepsilon = 3$ if $\ell = 2$. Then 
$$\gstspdim{c + m}{\ell}{k_r} \le \max(0,r-m+1)\quad \text{for $m \ge 1$},$$ and
\[\gstspdim{c + m}{\ell}{F} \leq
\begin{cases}
\frac12 r(r-1) + r\varepsilon + \gstspdim{c + 1}{\ell}{k(x)}
&\text{for $m=1$,} \\
\frac12(r - m + 1)(r - m) + (r - m + 2)\varepsilon &\text{for $2 \leq m \leq r+1$,} \\
0 & \text{for $m > r + 1$}
\end{cases}\]
for any one variable function field $F$ over $k_r$.
The same assertions hold for the stable splitting dimension.
\end{thm}

\begin{proof}
For the first assertion,  we have $\cd_\ell(k_j) = c + j$ for $j \ge 0$ by 
applying \cite[Proposition~II.4.3.12]{SerreGalois} inductively.  Thus
$\gstspdim{c + m}{\ell}{k_r} = 0$ if $m \ge r+1$, as asserted in that case.  On the other hand, if $m \le r$ then $\gstspdim{c+m}{\ell}{k_{m-1}} = 0$. 
Hence $\gstspdim{c + m}{\ell}{k_r} \le r-m+1$ by applying Lemma~\ref{discrete index bootstrap} to the sequence of fields $k_{m-1},\dots,k_r$.  

For the second assertion, again it suffices to consider the case when $F=k_r(x)$. Note that the
case $m > r + 1$ follows from the fact that $\cd_\ell(k_r(x)) = c + r + 1$ by \cite[Proposition~II.4.2.11]{SerreGalois}. The case $m = r + 1$ follows from Theorem~\ref{main thm}, using the fact that $\gstspdim{c + m}{\ell}{k_{r-1}(x)} = 0 = \gstspdim{c + m}{\ell}{k_{r - 1}}$ because of the cohomological dimension of these fields. 

For the case $2 \leq m \leq r$, observe that $\gstspdim{c+m}{\ell}{k_{m-1}} = 0 = \gstspdim{c+m}{\ell}{k_{m-2}(x)}$ because $\cd(k_{m-1}) = c+m-1 = \cd({k_{m-2}}(x))$, and similarly $\gstspdim{c+m}{\ell}{k_{m-2}} = 0$.  Thus Theorem~\ref{main thm} yields $\gstspdim{c+m}{\ell}{k_{m-1}(x)} \le \varepsilon$.  Now 
write $k'=k_{m-1}$ and $k'_j = k_{m-1+j}$.  Thus $k_r=k'_{r+1-m}$.
Applying Theorem~\ref{new iterative} with $k', c+m, r-m+1$ playing the roles of $k,i,r$ there, we have 
$\gstspdim{c+m}{\ell}{k_r(x)} \le \varepsilon + \frac{r-m+1}{2}(r-m+1 + 2 \!\cdot\! 0 + 2\varepsilon - 1) = \frac12(r - m + 1)(r - m) + (r - m + 2)\varepsilon$.

For $m=1$, we have $\gstspdim{c+1}{\ell}{k}=0$ since $\cd_\ell(k)=c$. Theorem~\ref{new iterative} with $i=c+1$ yields 
$\gstspdim{c+1}{\ell}{k_r(x)} \le  \gstspdim{c+1}{\ell}{k(x)}+ \frac{r}{2}(r + 2 \!\cdot\! 0 + 2\varepsilon - 1)
= \frac12 r(r - 1) + r\varepsilon + \gstspdim{c+1}{\ell}{k(x)}$.

The same proof shows the assertions on the stable splitting dimension, using the corresponding assertions in Lemma~\ref{discrete index bootstrap}, Theorem~\ref{main thm}, and Theorem~\ref{new iterative}.
\end{proof}

\begin{rem} \label{iterative rk}
\renewcommand{\theenumi}{\alph{enumi}}
\renewcommand{\labelenumi}{(\alph{enumi})}
\begin{enumerate}
\item \label{cd gssd sgf}
The bounds on $\gstspdim{i}{\ell}{k_r(x)}$ also apply to $\gstspdim{i}{\ell}{F}$ for any finite extension~$F$ of~$k_r(x)$, since the generalized stable $i$-splitting dimension either stays the same or decreases upon passing to a finite extension.
\item  \label{ssd iterative bounds}
In the case of $\stspdim{i}{\ell}{k_r(x)}$, the bounds given in 
Theorem~\ref{cd iterative sgf} are not in general sharp. For example, consider the field $k_r=\mbb C((s_1))\cdots((s_r))$ for $r \ge 1$, and let $\ell$ be a prime.  Then Theorem~\ref{cd iterative sgf} says that $\stspdim{2}{\ell}{k_r(x)} \le \gstspdim{2}{\ell}{k_r(x)} \le
\frac{1}{2}(r-1)(r-2)+ r\varepsilon$, with $\varepsilon = 2$ (resp., $3$) if $\ell \ne 2$ (resp., $=2$).  But according to \cite[Corollary~5.7]{HHK}, $\stspdim{2}{\ell}{k_r(x)} \le r$, which is smaller.
\item \label{up and down}
Theorem~\ref{cd iterative sgf} shows that if $k$ is fixed and $F$ is a one-variable function field over $k_r$ as above, then our bound on $\gstspdim{i}{\ell}{F}$ (resp., $\gstspdim{i}{\ell}{k_r}$) depends only on $r-i$ for $i > \cd_\ell(k)+1$ (resp., for $i > \cd_\ell(k)$); moreover the bound increases with $r$ and decreases with $i$ (and similarly for ${\rm ssd}_\ell^i$).  More precisely, as $i$ increases, our bound on $\gstspdim{i}{\ell}{k_r}$ decreases linearly until it reaches $0$, and our bound on $\gstspdim{i}{\ell}{k_r(x)}$ decreases quadratically; and the same happens as $r$ decreases.  For numerical examples, see the discussion following Proposition~\ref{number field ex}.
\item \label{vcd rk}
Suppose more generally that $k$ is a field with virtual $\ell$-cohomological dimension equal to~$c$; i.e., there is a finite field extension $k'/k$ such that $\cd_\ell(k')=c$.  Let $F$ be a one-variable function field over $k_r$, and let $F' = Fk'$.  Then for $i \ge c+1$, the value of $\gstspdim{i}{\ell}{F'}$ is bounded via the above theorem, and we have that $\gstspdim{i}{\ell}{F} \le v_\ell + \gstspdim{i}{\ell}{F'}$, where $v_\ell$ is the $\ell$-adic valuation of $[k':k]$.  
\end{enumerate}
\end{rem}

\section{Splitting for arithmetic surfaces} 

We have so far focused on the splitting of cohomology classes $\alpha \in H^i(F,\mu_\ell^{\otimes i-1})$ in the case of a semiglobal field $F$; i.e., a one-variable function field over a complete discretely valued field.  We can also consider the case of one-variable function fields $F$ over a global field.  Such a field $F$ has a model which is a two-dimensional regular integral scheme that is projective over either a finite field or the ring of integers of a number field (of relative dimension one).  In the latter case, there is the following splitting result when $i=3$ and $\ell=2$, due to a theorem of Suresh.

\begin{thm} \label{arith deg 2 splitting}
Let $\mathscr X$ be a two-dimensional regular integral scheme that is projective over the ring of integers of a number field.  Let $F$ be the function field of $\ms X$, and let 
$\gamma_1,\dots,\gamma_N \in H^3(F,\mu_2^{\otimes 2})$.  Then there is a degree two field extension of $F$ that splits each $\gamma_j$.
\end{thm}

\begin{proof}
Theorem~3.2 of~\cite{Suresh1} asserts that there exist $f \in F^\times$ and $\beta_j \in H^2(F,\mu_2)$ for $j=1,\dots,N$ such that $\gamma_j = (f) \cup \beta_j$ for all $j$.  Thus every $\gamma_j$ is split by the degree two extension $F(f^{1/2})$ of $F$.
\end{proof}

In the remainder of this section, our goal is to treat the analogous
situation for the function field $F$ of a regular projective surface over a finite field, with $\ell \ne \cha(F)$.  Specifically, in Theorem~\ref{deg ell splitting}, we show that a finite set of elements in $H^3(F,\mu_\ell^{\otimes 2})$ can all be split by some extension of degree $\ell$.  This will then be used in the next section to obtain values of ${\rm gssd}$ in situations related to global function fields, building also on the previous sections.  We first need some preliminary results.

\begin{lem} \label{Kummer}
Let $\ms X$ be a normal integral scheme whose function field $F$ contains a primitive $\ell$-th root of unity for some prime number $\ell$.  Let $P_1,\dots,P_r$ be closed points of $\ms X$ whose residue fields are finite of order prime to $\ell$.
Then there is a Galois field extension $L/F$ of degree $\ell$ such that the normalization $\ms Y$ of $\ms X$ in $L$ has the property that the fiber of $\ms Y \to \ms X$ over each $P_i$ is \'etale and consists of a single closed point of $\ms Y$.
\end{lem}

\begin{proof}
Choose an affine open subset $U = \Spec(R)$ of $\ms X$ that contains the points $P_i$, and let $\frak m_i$ be the maximal ideal of $R$ corresponding to $P_i$.  Let $k_i'$ be the unique degree $\ell$ field extension of the finite field $k_i := \kappa(P_i)$.  By the hypothesis on $F$, the field $k_i$ contains a primitive $\ell$-th root of unity; and so $k_i'/k_i$ is a Kummer extension, given by extracting an $\ell$-th root of some element $a_i \in k_i$ that is not an $\ell$-th power in $k_i$.  Since the maximal ideals $\frak m_i$ are pairwise relatively prime, by the Chinese Remainder Theorem there is an element $a \in R$ whose reduction modulo $\frak m_i$ is $a_i$ for all~$i$.  Here $a$ is not an $\ell$-th power in $F$.  The reduction of $S := R[x]/(x^\ell - a)$ modulo $\frak m_iS$ is $k_i'$ for all $i$, and so its fraction field $L$ has the asserted property.
\end{proof}

\begin{lem} \label{Kummer regular}
Let $R$ be a regular local ring of dimension two with fraction field $E$, and let $f,g$ be a system of parameters at the maximal ideal of $R$.  Let $L/E$ be a cyclic field extension whose degree $\ell$ is a prime number that is unequal to the residue characteristics of $R$ and such that $E$ contains a primitive $\ell$-th root of unity.  
Let $S$ be the normalization of $R$ in $L$, and suppose that $S[f^{-1}]$ is unramified over $R[f^{-1}]$.  Then $S$ is regular.
\end{lem}

\begin{proof}
By the hypotheses, $L/E$ is a Kummer extension; i.e., $L = E[h^{1/\ell}]$ for some $h \in E$ that is not an $\ell$-th power.  After multiplying $h$ by an $\ell$-th power, we may assume that $h \in R$ and so $h^{1/\ell} \in S$.  Since the regular local ring $R$ is a unique factorization domain (by \cite[Theorem~5]{AuBu}), we may write $h=uh_1^{d_1}\cdots h_n^{d_n}$ with $n \ge 0$, where $u$ is a unit in $R$, the elements $h_i \in R$ are irreducible and define distinct height one primes, and each $d_i \ge 1$.  
After dividing $h$ by an $\ell$-th power, we may assume that $1 \le d_i < \ell$ for all $i$.  Since the residue characteristics of $R$ are unequal to $\ell$, the subring $R[h^{1/\ell}] \subseteq S$ is ramified over $R$ precisely over the primes $(h_i)$.  

If $n=0$ then the subring $R[h^{1/\ell}] = R[u^{1/\ell}] \subseteq L$ is finite \'etale over $R$, and hence regular.  So it is equal to its normalization; i.e., its integral closure in its fraction field $L$, which is $S$.  Thus $S$ is regular.
Alternatively, if $n>0$, then since $S[f^{-1}]$ is unramified over $R[f^{-1}]$, and since $f,h_1$ are both irreducible in $R$, it follows that $n=1$ and $h_1=vf$ for some unit $v \in R$.
Since $d_1$ and $\ell$ are relatively prime, there exist integers $a,b>0$ with $ad_1-b\ell=1$.  Hence $h^a=u^a v^{ad_1} f^{1+b\ell}$; and so $S$ contains an $\ell$-th root of $u^a v^{ad_1} f^{1+b\ell}$ and thus also of $f_1:=u^a v^{ad_1} f$.  The elements $f_1^{1/\ell}, g$ form a system of parameters for the subring $S'=R[f_1^{1/\ell}] \subseteq S$, which is therefore regular.  Since $f_1 = h^a/f^{b\ell}$ is not an $\ell$-th power in $E$, the fraction field of $S'$ has degree $\ell$ over $E$ and so is equal to $L$, the fraction field of $S$.  But $S$ is the normalization of $R$ in $L$, and hence also that of $S'$ in $L$.  Since the regular ring $S'$ is normal, $S=S'$, and so $S$ is regular.
\end{proof}

\begin{rem}
The conclusion of Lemma~\ref{Kummer regular} fails if $\cha(R)=0$ but $R$ has primes of residue characteristic $\ell$, even though $L/E$ is Kummer.  For example, let $R=\mbb Z_2[[x,y]]/(xy-2)$, for which $x,y$ form a system of parameters.  Let $E$ be the fraction field of $R$, take $\ell=2$, let $h=2y^2+1$, and write $L = E[h^{1/2}] = E[w]/(w^2-h)$.  Here $h$ is a unit in $R$; but $R[h^{1/2}]$ is not \'etale over $R$, being purely inseparable over the primes $(x)$ and $(y)$, where the residue characteristic is $2$.  Moreover $R[h^{1/2}]$ is not normal; its normalization $S$ (in its fraction field~$L$) is obtained by adjoining to $R[h^{1/2}]$ the element $z= (w+1)/y \in L$.
As an abstract ring, $S = R[z]/(z^2-xz-2)$.
This ring is ramified precisely over $(x)$, but it is not regular, having a singularity at its maximal ideal $(x,y)$.  
This phenomenon, which is contrary to the situation of Lemma~\ref{Kummer regular}, leads to difficulties in treating the analog of Theorem~\ref{deg ell splitting} in the case of a projective scheme of relative dimension one over the spectrum of the ring of integers of a number field, with general $\ell$.
\end{rem}

The following known result will be useful in proofs below, and we state it for ease of citation.

\begin{lem} \label{square}
Let $K'/K$ be an extension of discretely valued fields
with residue field extension $k'/k$ and ramification index $e$.  Let $\ell \ne \cha(k)$ be a prime and let $i$ be a non-negative integer. 
Then the diagram
                     \[\begin{tikzcd}
H^{i+1}(K, \mu_\ell^{\otimes i})\arrow{r}{{\rm res}}\arrow{d}
&H^i(k, \mu_\ell^{\otimes i-1})\arrow{d}{e}\\
H^{i+1}(K', \mu_\ell^{\otimes i})
\arrow{r}{{\rm res}}
&H^i(k', \mu_\ell^{\otimes i-1})
\end{tikzcd}\]
commutes, where the horizontal arrows are given by residues, the left hand vertical arrow is the natural map, and the right hand vertical arrow is the product of $e$ with the natural map. 
\end{lem}

\begin{proof}
This is a special case of \cite[Proposition~II.8.2, p.~19]{GMS}.
\end{proof}

These next lemmas will be used to verify properties needed 
in the proof of Theorem~\ref{deg ell splitting}, concerning the ramification and splitting behavior of cohomology classes under pullback.

\begin{lem} \label{unram extension}
Let $\ms Z \to \ms X$ be a morphism of regular integral two-dimensional schemes, with function fields $L/F$. 
Let $\ell$ be a prime number unequal to the residue characteristics at the points of $\ms X$ and $\ms Z$, and let $\gamma \in H^3(F,\mu_\ell^{\otimes 2})$.  If $\gamma$ is unramified on $\ms X$ then its restriction $\gamma' \in H^3(L,\mu_\ell^{\otimes 2})$ is unramified on $\ms Z$.
\end{lem}

\begin{proof}
Let $\zeta$ be a codimension one point of $\ms Z$.  We wish to show that the residue of $\gamma'$ at $\zeta$ is trivial.  Let $\xi$ be the image of $\zeta$ in $\ms X$.  Thus $\xi$ has codimension one or two on $\ms X$.  In the former case, $\gamma$ has trivial residue at $\xi$, hence $\gamma'$ has trivial residue at $\zeta$ by Lemma~\ref{square}.

Now assume that $\xi$ has codimension two in~$\ms X$.  
The rows in the commutative diagram
\[\begin{tikzcd}
H^3(\ms O_{\ms Z,\zeta}, \mu_\ell^{\otimes 2})\arrow{r}&H^3(L,\mu_\ell^{\otimes 2}))\arrow{r}{{\rm res}}
&H^2(\kappa(\zeta), \mu_\ell)\\
H^3(\ms O_{\ms X,\xi}, \mu_\ell^{\otimes 2})\arrow{r}\arrow{u}&H^3(F,\mu_\ell^{\otimes 2}))
\arrow{r}{{\rm res}}\arrow{u}
&\prod_{x \in \Spec(\ms O_{\ms X,\xi})^{(1)}}
H^2(\kappa(x), \mu_\ell)
\end{tikzcd}\]
are complexes, and the lower row is exact by \cite[Proposition~6]{Sakagaito}.  Since $\gamma$ is unramified on $\ms X$, it is the image of an element $\til \gamma \in H^3(\ms O_{\ms X,\xi}, \mu_\ell^{\otimes 2})$, by the exactness.  
Let $\til \gamma' \in H^3(\ms O_{\ms Z,\zeta}, \mu_\ell^{\otimes 2})$ be the image of $\til \gamma$. 
So the image  of $\til \gamma'$ in $H^3(L,\mu_\ell^{\otimes 2})$ is unramified at $\zeta$.  This latter image is $\gamma'$ by commutativity of the above square , so the conclusion follows.
\end{proof}

Given a field $L$, an arbitrary prime $\ell$, and non-negative integers $i,j$, Kato defined an abelian group $H^i(L,\mbb Z/\ell\mbb Z(j))$ that agrees with $H^i(L,\mu_\ell^{\otimes j})$ in the case that $\cha(L)\ne \ell$ (see \cite[page~143]{Kato}).  Moreover, as stated there, $H^2(L,\mbb Z/\ell\mbb Z(1))$ is just the $\ell$-torsion subgroup of $\Br(L)$, and $H^1(L,\mbb Z/\ell\mbb Z)$ is the same as $\Hom_{\rm cont}(\Gal(L^{\rm ab}/L),\mbb Z/\ell\mbb Z)$.

\begin{lem} \label{bl up reg pt}
Let $\ms X$ be a two-dimensional regular integral scheme that is projective over either a finite field or the ring of integers of a number field that we assume to be totally imaginary.  Let $\gamma \in H^3(F,\mbb Z/\ell\mbb Z(2))$ for some prime number $\ell \ne \cha(F)$, where $F$ is the function field of $\ms X$.  Let 
$C$ be a codimension one subscheme of $\ms X$ that contains the closures of the codimension one points of $\ms X$ where $\gamma$ is ramified. 
Consider the blow-up $\til{\ms X} \to \ms X$ of $\ms X$ at a finite set of regular points of $C$.
Then $\gamma$ is unramified at the generic point of each exceptional divisor of the blow-up.
\end{lem}

\begin{proof}
The field $F$ has no ordered field structure, and so the hypotheses of \cite[Corollary to Theorem~0.7]{Kato} are satisfied.  That result then provides an exact sequence
\[0 \to H^3(F, \mathbb Z/\ell\mbb Z(2)) \to \bigoplus_{\eta \in \til{\ms X}_1} H^2(\kappa(\eta), \mathbb Z/\ell\mbb Z(1))
\to \bigoplus_{x \in \til{\ms X}_0} H^1(\kappa(x), \mathbb Z/\ell\mbb Z(1)) \to \mbb Z/\ell\mbb Z \to 0,\]
where the maps are given by residues, and where $\til{\ms X}_i$ is the set of dimension $i$ points on $\til{\ms X}$.  

Let $\xi$ be one of the closed points of $\ms X$ that is blown up.
By the regularity hypotheses, the exceptional divisor $E$ over $\xi$ is a copy of $\mbb P^1_{\kappa(\xi)}$ that meets the proper transform of $C$ at a single point $\til\xi$. 
Consider any closed point $x_0 \in E$ other than $\til\xi$.  Then except for the generic point $\eta_0 \in \til{\ms X}_1$ of $E$, the class $\gamma$ is unramified at the dimension one points of $\til{\ms X}$ whose closure contains $x_0$.  So only one term in $\bigoplus_{\eta \in \til{\ms X}_1} H^2(\kappa(\eta), \mathbb Z/\ell\mbb Z(1))$ contributes to the image of $\gamma$ in $H^1(\kappa(x_0), \mathbb Z/\ell\mbb Z(1))$; viz., the one arising from $\eta_0 \in \til{\ms X}_1$.  Since the image of $\gamma$ in $\bigoplus_{x \in \til{\ms X}_0} H^1(\kappa(x), \mathbb Z/\ell\mbb Z(1))$ is $0$, it follows that the contribution of that one term is also zero; i.e., $\alpha := \res_{\eta_0}(\gamma)$ is unramified at $x_0$, where $x_0$ is an arbitrary closed point of $E$ other than $\til\xi$.

The complement of the $\kappa(\xi)$-point $\til\xi$ of $E \cong \mbb P^1_{\kappa(\xi)}$
is isomorphic to the affine line over $\kappa(\xi)$.  Since $\alpha$ is unramified over that complement, it is induced from a class in $H^2(\kappa(\xi),\mathbb Z/\ell\mbb Z(1))$ by \cite[Theorem~III.9.3, p.24]{GMS}.  
But $H^2(\kappa(\xi),\mathbb Z/\ell\mbb Z(1))$ is the $\ell$-torsion subgroup of $\Br(\kappa(\xi))$, which is trivial since $\kappa(\xi)$ is a finite field.  Hence $\alpha=0$. 
\end{proof}

\begin{lem} \label{unram over C}
Let $\ell$ be a prime number, and 
let $\mathscr X$ be a two-dimensional regular integral scheme that is projective over either a finite field or the ring of integers of a number field that we assume to be totally imaginary if $\ell=2$.
Let $\ms Y$ be the normalization of $\ms X$ in a degree $\ell$ separable field extension $L/F$, let $C \subset \ms X$ be a regular connected curve with function field $\kappa(C)$, 
and let $\alpha$ be an $\ell$-torsion element of $\Br(\kappa(C))$.  Suppose that at every closed point $P$ of $C$ at which $\alpha$ is ramified, $\pi:\ms Y \to \ms X$ is \'etale and $\pi^{-1}(P)$ is a single point. 
If $\eta \in \ms Y$ lies over the generic point of~$C$, then the pullback $\alpha_{\kappa(\eta)}$ is split.
\end{lem}

\begin{proof}
Let $\mc P$ be the set of closed points of $C$ where $\alpha$ is ramified.  Let $D \subseteq \pi^{-1}(C)$ be the closure of $\eta$, with normalization $\tilde D \to D$.
The pullback $\alpha_{\kappa(\eta)} \in \Br(\kappa(\tilde D))=\Br(\kappa(D))$ of $\alpha \in \Br(\kappa(C))$
is unramified away from $\pi^{-1}(\mc P)$. Since $\pi$ is \'etale over each $P \in \mc P$, so is $D \to C$; hence 
$D$ is regular there and $\til D \to D$ is an isomorphism over ${\mc O}_P(C)$.
So $\tilde D \to C$ is \'etale over $P$, with just one point in the fiber.  The residue field extension there
is the unique degree $\ell$ extension of the finite field $\kappa(P)$,
so it agrees with the residue $\res_P(\alpha) \in H^1(\kappa(P),\mbb Z/\ell \mbb Z)$ of the $\ell$-torsion class $\alpha$ at the ramified point $P$.
Thus $\alpha_{\kappa(\eta)}$ is unramified at each point over $\mc P$, hence at every point of $\til D$.
So the $\ell$-torsion class $\alpha_{\kappa(\eta)}$ lies in $\Br(\til D)$ by \cite[Theorem~3.7.7]{CTSk}.  But 
$\Br(\til D)$ has trivial $\ell$-torsion; e.g., see \cite[Remarque~III.2.5(b)]{GpBr} if $\til D$ is a smooth projective curve over a finite field, and see \cite[Proposition~III.2.4]{GpBr} if instead $\til D = \Spec(\mc O_K)$ for a number field $K$ that is totally imaginary if $\ell=2$. Hence $\alpha_{\kappa(\eta)}$ is split.
\end{proof}

We now come to the main result of this section.

\begin{thm} \label{deg ell splitting}
Let $\mathscr X$ be a two-dimensional regular integral scheme that is projective over a finite field. 
Let $F$ be the function field of $\ms X$.  Assume that $F$ contains a primitive $\ell$-th root of unity for some prime $\ell$,
and let $\gamma_1,\dots,\gamma_N \in H^3(F,\mbb Z/\ell \mbb Z(2))$.  Then there is a field extension of $F$ of degree $\ell$ that splits each $\gamma_j$.
\end{thm}

\begin{proof}
Let $C$ be an effective divisor on $\ms X$ that contains all the codimension one points of $\ms X$ at which at least one of the classes $\gamma_j$ is ramified. By \cite[p.~193]{Lip75}, there is a blow-up $\ms X'$ of $\ms X$ such that the total transform of $C$ is a strict normal crossings divisor (i.e., it has only normal crossings and its components are regular).  So after replacing $\ms X$ by $\ms X'$, we may assume that $C$ itself satisfies this condition.  Let $C_1,\dots,C_m$ be the irreducible components of $C$, with function fields $\kappa(C_i)$, and let $\alpha_{i,j} \in \Br(\kappa(C_i))$ be the residue of $\gamma_j$ at the generic point  $\xi_i$ of $C_i$.  Thus $\alpha_{i,j}$ is $\ell$-torsion.  

Let $\mc P$ 
be a finite set of closed points of $\ms X$ with at least one point on each $C_i$, such that $\mc P$ contains all the singular (normal crossing) points of $C$ and all the points at which any of the classes $\alpha_{i,j}$ is ramified.  (In fact, all of these ramification points are singular points, by the exact sequence at the beginning of the proof of Theorem~\ref{bl up reg pt}.)
Let $L/F$ be the cyclic field extension given by Lemma~\ref{Kummer} applied to the points of $\mc P$.  
Let $\ms Y \to \ms X$ be the normalization of $\ms X$ in $L$, and let $B$ be its branch locus. 
Over each point of $\mc P$ the morphism $\ms Y \to \ms X$ is \'etale and the fiber consists of a single point; hence the same holds for the generic points $\xi_i$ of the curves $C_i$, and moreover the divisor $B$ does not pass through any point of $\mc P$.
There is then a blow-up $\til{\ms X} \to \ms X$, centered only at points where $B \cup C$ has a singularity other than a normal crossing,
such that the total transform of $B \cup C$ is a strict normal crossing divisor.
Since the singular points of $C$ lie in $\mc P$, none of those points lie on $B$ and none of them are among the points that are blown up.  So the proper transform $\til C$ of $C$ maps isomorphically onto $C$, with its irreducible components mapping isomorphically onto respective components $C_i$ of $C$.

We now reduce to the case that $B \cup C$ is itself a strict normal crossing divisor.  To do this, first observe that none of the cohomology classes $\gamma_j$ are ramified at any of the exceptional divisors of $\til{\ms X}\to \ms X$, by Lemma~\ref{unram extension} applied to the complement $U \subseteq \ms X$ of $C$, in the case of an exceptional divisor lying over a point that does not lie on $C$; and by Lemma~\ref{bl up reg pt} in the case of an exceptional divisor lying over a (regular) point of $C$.  Thus the proper transform $\til C$ of $C$ contains all the codimension one points of $\til{\ms X}$ at which at least one of the classes $\gamma_j$ is ramified.  
Let $\til{\ms Y} \to \til{\ms X}$ be the normalization of $\til{\ms X}$ in $L$; its branch locus is contained in the total transform of $B$. So replacing $\ms Y \to \ms X$ by $\til{\ms Y} \to \til{\ms X}$, replacing $C$ by its (isomorphic) proper transform $\til C$ and similarly for its irreducible components $C_i$, replacing $\mc P \subset C$ by its inverse image in $\til C$, and
replacing $B$ by the branch locus of $\til{\ms Y} \to \til{\ms X}$ (which is contained in the total transform of the original $B$),
we may assume that $B \cup C$ is a strict normal crossing divisor in $\ms X$.  In doing so, we retain the property that the cohomology classes $\gamma_j$ are ramified only at codimension one points of $\ms X$ that lie on (the new) $C$.

Our next step is to show that the given cohomology classes $\gamma_j$ are each unramified at every codimension one point of $\ms Y$.  To see this, note that since the given cohomology classes $\gamma_j$ are unramified at the codimension one points on the complement $U \subseteq \ms X$ of $C$, they remain unramified at the codimension one points on its inverse image $V \subseteq \ms Y$ by Lemma~\ref{square}.  
The other codimension one points of $\ms Y$ lie over the generic points $\xi_i$ of $C_i$ for $i=1,\dots,m$.  
As noted above, there is a unique point $\eta_i$ in $\ms Y$ over each $\xi_i$.  Now $\ms Y \to \ms X$ is \'etale over the points of $\mc P$ with each of those fibers consisting of a single point; so this holds in particular at the points where each $\alpha_{i,j}$ is ramified.  It then follows from Lemma~\ref{unram over C} that $(\alpha_{i,j})_{\eta_i}$ is split.   
That is, $\gamma_j$ is unramified at the points $\eta_i \in \ms Y$ lying over the generic points of the curves $C_i$, as well as at the other codimension one points of $\ms Y$; and that completes this step.   

Next, we claim that 
$\ms Y$ is regular at every closed point $Q$ lying over a point $P$ of $C$.  To see this, note that $\ms Y$ is regular at $Q$ if $P$ is not a point of $B$, since $\ms Y \to \ms X$ is \'etale there and $\ms X$ is regular.  Now suppose that $P \in B$.  Then $P$ is a nodal point of $B \cup C$, and is a regular point of $B$ and of $C$, lying on a unique irreducible component of each.  These components are respectively defined in $\mc O_{\ms X,P}$ by elements $f,g$ that form a system of parameters.  By Lemma~\ref{Kummer regular},  $\mc O_{\ms Y,Q}$ is regular, proving the claim.

Every singular point of $\ms Y$ lies in $V$ by the above claim, 
and the two-dimensional normal scheme $\ms Y$ has only finitely many singular points.
Thus there is a blow-up $\ms Z \to \ms Y$ centered at those points of $V$, with $\ms Z$ regular.  This is
an isomorphism away from those finitely many points, and Lemma~\ref{square} implies that the classes $\gamma_j$ are unramified at every codimension one point of $\ms Z$ that lies over a codimension one point of $\ms Y$.  The only other codimension one points of $\ms Z$ are the generic points of the exceptional divisors of the blow-up $\ms Z \to \ms Y$, which lie over closed points of $V$.  Let $W$ be the inverse image of $U \subseteq \ms X$ (or equivalently, of $V \subseteq \ms Y$) in $\ms Z$.  
Applying Lemma~\ref{unram extension} to $W \to U$, we find that the classes  $\gamma_j$ are unramified at the codimension one points of $W$, and in particular at the exceptional divisors of $\ms Z \to \ms Y$.
Since $\ms Z$ is regular with function field $L$,  \cite[Corollary to Theorem~0.7]{Kato} 
asserts that the residue map $H^3(L, \mbb Z/\ell\mbb Z(2)) \to \bigoplus_{\zeta \in \ms Z_1} H^2(\kappa(\zeta), \mbb Z/\ell\mbb Z(1))$ is injective, where $\ms Z_1$ is the set of dimension one points of $\ms Z$.  Hence the pullback of each $\gamma_j$ to $H^3(L, \mathbb Z/\ell\mbb Z(2))$ is trivial, as needed.
\end{proof}

\section{Applications}

This section gives concrete applications of our bound.
We start with an example involving $3$-dimensional fields over the complex numbers. A result of de Jong (\cite{deJ}) shows that for the function field of a complex algebraic surface, the index of a Brauer class (that is, an element in degree $2$ cohomology) must equal its period. In contrast, bounds for the index of a degree~$3$ cohomology class on the function field of a complex threefold are not known. On the other hand, if we consider a somewhat simpler 3-dimensional field $F$, namely a finite extension of the field $\mathbb C(x, y)((t))$, it follows (for example from Lemma~\ref{discrete index bootstrap}) that a class in $H^3(F, \mu_\ell^{\otimes 2})$ will have index at most $\ell$. If $F$ is a finite extension of $\mathbb C(y)((t))(x)$, the arithmetic is more subtle.
Using \cite{deJ} to show $\stspdim{2}{\ell}{\mathbb C(x,y)} \leq 1$, Theorem~\ref{main thm} gives that $\stspdim{2}{\ell}{\mathbb C(y)((t))(x)} \leq 3$ or $4$, depending on the parity of $\ell$. On the other hand, de Jong's theorem does not give us information about $\gstspdim{2}{\ell}{\mathbb C(x,y)}$,
and hence the methods of \cite{gosavi2019generalized} and Proposition~\ref{inductive} do not give bounds on the index of degree~$3$ cohomology classes for such fields.
Using our new results, we obtain the following bounds for degree~$3$ cohomology:

\begin{prop} \label{cx fn fld ex}
Let $k={\mathbb C}(\ms Y)$ be the function field of a complex curve. Let $\ell$ be a prime.
\begin{itemize}
\item[(a)]
If $F$ is a one-variable function field over $k((s))$, then $\gstspdim{3}{\ell}{F} \le 2$ if $\ell$ is odd and $\gstspdim{3}{\ell}{F} \le 3$ if $\ell = 2$.
\item[(b)] More generally, if $F_r$ is a one-variable function field over $k((s_1)) \cdots ((s_r))$ for $r > 0$, then $\gstspdim{3}{\ell}{F} \le (r^2+r+2)/2$ if $\ell$ is odd and $\gstspdim{3}{\ell}{F} \le (r^2+3r+2)/2$ if $\ell = 2$.
\end{itemize}
\end{prop}
\begin{proof}
Note that $k$ and $k(x)$ have cohomological dimension $1$ and $2$ respectively, and thus $\gstspdim{3}{\ell}k=\gstspdim{3}{\ell}{k(x)}=0$. The first statement now follows directly from Theorem~\ref{main thm}. The second statement is by Theorem~\ref{cd iterative sgf} (with $m=2$).
\end{proof}
In the situation above, Theorem~\ref{cd iterative sgf} also gives bounds for $\gstspdim{i}{\ell}{F_r}$ when $3<i<r+3$; e.g., $\gstspdim{4}{\ell}{F_r}\leq (r^2-r+2)/2$ if $\ell$ is odd and $\gstspdim{4}{\ell}{F_r} \le (r^2+r)/2$ if $\ell = 2$. As $i$ increases, $\gstspdim{i}{\ell}{F_r}$ decreases, and becomes~$0$ for $i\geq r+3$. Bounds for $\gstspdim{2}{\ell}{F_r}$ were given in \cite{gosavi2019generalized}.

\medskip

We now move on to a class of examples related to global residue fields. Information about the period-index problem for degree $2$ cohomology classes when $F$ is a one-variable function field over a number field has been highly sought after. As of yet, bounds of this type are only known contingent upon conjectures of Colliot-Th\'el\`ene \cite{LPS}. Remarkably, the work of Lieblich \cite{Lie15} has shown that the index divides the square of the period in the case of a function field~$F$ of a surface over a finite field, giving $\stspdim{2}{\ell}{F} \leq 2$ in this case. Nevertheless, in neither situation do we have information on $\gstspdim{2}{\ell}{F}$, and so again we are unable to apply  \cite{gosavi2019generalized} or Proposition~\ref{inductive} to obtain bounds on the index of a cohomology class of degree higher than~$3$. On the other hand, degree $3$ cohomology over such fields is much more directly tractable, as was highlighted in the work of Kato \cite{Kato}. 
Building on Theorems~\ref{arith deg 2 splitting} and~\ref{deg ell splitting} above together with our previous results, we obtain Proposition~\ref{globalsplitting}, Proposition~\ref{number field ex}, and the numerical examples that follow.  First we state a lemma.

\begin{lem} \label{nontriv_h3}
Let $k$ be a global field, let $E$ be the function field of a regular projective $k$-curve $C$, and let $\ell$ be a prime unequal to $\cha(k)$. Then $H^3(E, \mu_\ell^{\otimes 2}) \neq 0$.
\end{lem}

\begin{proof}
Let $P \in C$ be any closed point, and let $k'$ be its residue field.  Since $k'$ is also a global field, 
the $\ell$-torsion subgroup $\Br(k')[\ell] \subseteq \Br(k')$ is non-trivial (e.g., by \cite[Corollary~6.5.3, Proposition~6.3.7]{GSz} in the function field case and \cite[Theorem~18.5]{Pie82} in the number field case).
The period and index of a non-trivial element $\alpha \in \Br(k')[\ell]$
both equal $\ell$ since $k'$ is a global field. By \cite[Theorem~3.11, Corollary~5.3]{Sa:RRC}, as $\ell$ is prime, we may lift $\alpha$ to an index 
$\ell$ class $\til \alpha \in \Br(\ms O_{C, P})[\ell] \subseteq \Br(E)[\ell] = H^2(E, \mu_\ell)$. Let $t \in \mc O_{C, P} \subset E$ be a uniformizer at $P$, and set $\beta = \til \alpha \cup (t) \in H^3(E, \mu_\ell^{\otimes 2})$. Then $\res_{v_P}(\beta) = \alpha \neq 0$ by \cite[Proposition~II.7.11, p.~18]{GMS}, using that the residue homomorphism $\res_{v_P}$ associated to the discrete valuation $v_P$ defined by $P$ is defined by passing through the completion (\cite[Section~II.7.13, p.~19]{GMS}). Hence $\beta \in H^3(E, \mu_\ell^{\otimes 2})$ is nonzero.
\end{proof}

\begin{prop} \label{globalsplitting}
Suppose $k$ is a global field.  If $k$ is a function field, choose a prime $\ell\neq \cha(k)$.  If $k$ is a number field, take $\ell=2$.  Let $E$ be
 a one-variable function field over $k$. 
Then $\spdim{3}{\ell}{E} = \stspdim{3}{\ell}{E} =\gstspdim{3}{\ell}{E}=1$. 
\end{prop}

\begin{proof}
By Lemma~\ref{nontriv_h3}, $H^3(E, \mu_\ell^{\otimes 2}) \neq 0$, hence $0 < \spdim{3}{\ell}{E} \le \stspdim{3}{\ell}{E} \le \gstspdim{3}{\ell}{E}$.   
Thus to show that $\spdim{3}{\ell}{E} = \stspdim{3}{\ell}{E}=\gstspdim{3}{\ell}{E} =1$, it suffices to prove that $\gstspdim{3}{\ell}{E}$ is at most $1$. Every finite extension of $E$ is of the same form (i.e., a one-variable function field over a global field).  So it suffices to consider classes in $H^3(E, \mu_\ell^{\otimes 2})$, and not separately treat classes over finite extensions $E'$ of $E$.  By Lemma~\ref{l-power}, we may also assume that $E$ contains a primitive $\ell$-th root of unity, since adjoining this element produces a field extension of degree prime to $\ell$. 

If $k$ is a function field, then the desired assertion
is now immediate from Theorem~\ref{deg ell splitting}.  In the case where $k$ is a number field and $\ell=2$, it is immediate from Theorem~\ref{arith deg 2 splitting}.
\end{proof}

Our next examples concern function fields over higher local fields whose residue field is a global field. Examples of such fields include 
$F=K(x)$ where $K={\mathbb Q}((s))$ or ${\mathbb F}_p(y)((s))$, or where $K$ is the $p$-adic completion of 
${\mathbb Q}_p(t)$, or where $K$ 
is a field of iterated Laurent series over one of these fields.

\begin{prop} \label{number field ex}
Let $k$ be a global field, and let $\ell\neq \cha(k)$ be a prime. In the number field case assume $\ell=2$.  Suppose we have a sequence of fields $k = k_0, k_1, \ldots, k_r$, with $r \ge 1$, where $k_j$ is a complete discretely valued field with residue field $k_{j-1}$ for all $j\geq 1$, and let $F$ be a one-variable function field over $k_r$.  Then
\begin{itemize}
\item if $\ell$ is odd, we have $\gstspdim{3}{\ell}{F} \leq 1 + \frac r 2 (r +3)$,
\item if $\ell$ is even, and $k$ has no real places, we have $\gstspdim{3}{\ell}{F} \leq 1 + \frac r 2 (r +5)$,
\item if $\ell$ is even, and $k$ has real places, we have $\gstspdim{3}{\ell}{F} \leq 2 + \frac r 2 (r +5)$.
\end{itemize}
\end{prop}

\begin{proof}
Since $F$ is a finite extension of $k_r(x)$, we have that $\gstspdim{3}{\ell}{F} \le \gstspdim{3}{\ell}{k_r(x)}$.  Hence it suffices to prove the assertion for $F=k_r(x)$.

If $k$ is a number field (and $\ell=2$), we can reduce to the case that $k$ has no real places by adjoining a square root of $-1$ if necessary.  This 
increases by $1$ the power of $\ell$ in the degree of the splitting extension, and so the bound on $\gstspdim{3}{\ell}{F}$ increases by $1$ (as in the assertion of the third case).  So we can now assume that the global field $k$ has no real places, and in particular that we are in one of the first two cases. 

In the notation of Theorem~\ref{new iterative} with $i=3$, we have $d=\gstspdim{3}{\ell}{k} = 0$, by \cite[Proposition~II.4.4.13]{SerreGalois} in the case of a totally imaginary number field, and by \cite[Corollary in II.4.2]{SerreGalois} in the global function field case. Moreover, $\delta = \gstspdim{3}{\ell}{k(x)}=1$ by Proposition~\ref{globalsplitting}. Theorem~\ref{new iterative} thus gives the desired bounds.
\end{proof}

In the situation of Proposition~\ref{number field ex}, if $k$ has no real places, then for $r=1,2,3$ we find $\gstspdim{3}{\ell}{F} \le 3,6,10$, respectively, if $\ell$ is odd; and 
$\le 4,8,13$, respectively, if $\ell=2$.  
Again, Theorem~\ref{cd iterative sgf} gives information on the higher cohomology groups. Note that $c=\cd_\ell(k)=2$ as in the above proof; moreover, $\gstspdim{3}{\ell}{k(x)}=1$ by Proposition~\ref{globalsplitting}. Hence for this field $F$ with $r=1,2,3$, Theorem~\ref{cd iterative sgf} yields that
$\gstspdim{4}{\ell}{F} \le 2,4,7$ respectively if $\ell$ is odd, and $\le 3,6,10$ respectively if $\ell=2$.  Observe that our bound for $\gstspdim{i}{\ell}{F}$ decreases as $i$ increases.  For example, if $r=3$ then
$\gstspdim{i}{\ell}{F} \le 10,7,4,2,0$ for $i=3,4,5,6,7$ if $\ell$ is odd, and $\le 13, 10, 6, 3, 0$ if $\ell=2$.
Note in particular the relationship between the bounds for $\gstspdim{i}{\ell}{F}$ as $i$ increases and those as $r$ decreases (and see Remark~\ref{iterative rk}(\ref{up and down}) for a further discussion).

On the other hand, if $k$ is a number field with a real place (and $\ell=2$), then 
the bounds each increase by $1$ as above.  For example, for $r=1,2,3$ in that case, we have $\gstspdim{3}{2}{F} \le 5,9,14$ and $\gstspdim{4}{2}{F} \le 4,7,11$, respectively.  And for $r=3$ in that case, $\gstspdim{i}{2}{F} \le 14, 11, 7, 4, 1$ for $i=3,4,5,6,7$.

\newcommand{\etalchar}[1]{$^{#1}$}
\def\cprime{$'$} \def\cprime{$'$} \def\cprime{$'$} \def\cprime{$'$}
  \def\cftil#1{\ifmmode\setbox7\hbox{$\accent"5E#1$}\else
  \setbox7\hbox{\accent"5E#1}\penalty 10000\relax\fi\raise 1\ht7
  \hbox{\lower1.15ex\hbox to 1\wd7{\hss\accent"7E\hss}}\penalty 10000
  \hskip-1\wd7\penalty 10000\box7}

\medskip

\noindent{\bf Author Information:}

\medskip
 
\noindent David Harbater\\
Department of Mathematics, University of Pennsylvania, Philadelphia, PA 19104-6395, USA\\
email: harbater@math.upenn.edu

\medskip

\noindent Julia Hartmann\\
Department of Mathematics, University of Pennsylvania, Philadelphia, PA 19104-6395, USA\\
email: hartmann@math.upenn.edu

\medskip

\noindent Daniel Krashen\\
Department of Mathematics, University of Pennsylvania, Philadelphia, PA 19104-6395, USA\\
email: dkrashen@math.upenn.edu

\medskip

\noindent The authors were supported on NSF grants DMS-1805439 and DMS-2102987 (DH and JH), and by NSF RTG grant DMS-1344994 and NSF grant DMS-1902144 (DK).

\end{document}